\documentclass[leqno, a4paper]{amsart}
\usepackage[utf8]{inputenc}

\usepackage{babelbib}
\usepackage[USenglish]{babel}
\usepackage{amsmath}
\usepackage{mathtools}
\usepackage{amssymb}
\usepackage{hyperref}
\usepackage{amsthm}
\usepackage[all]{xy}

\theoremstyle{plain}
\newtheorem{lem}{Lemma}[section]
\newtheorem{prop}[lem]{Proposition}
\newtheorem{cor}[lem]{Corollary}
\newtheorem{thm}[lem]{Theorem}
\newtheorem{conj}[lem]{Conjecture}
\newtheorem*{thm*}{Theorem}
\newtheorem*{prop*}{Proposition}
\newtheorem*{cor*}{Corollary}
\newtheorem*{lem*}{Lemma}

\newtheorem*{thmt*}{Theorem (tame case)}
\newtheorem*{thmw*}{Theorem (wild case)}

\theoremstyle{definition} 
\newtheorem{ex}[lem]{Example}
\newtheorem*{nota}{Notation}
\newtheorem{rem}[lem]{Remark}

\newtheorem{defn}[lem]{Definition}

\DeclareMathOperator{\Aut}{\mathrm{Aut}}

\DeclareMathOperator{\Char}{\mathrm{char}}
\DeclareMathOperator{\Spec}{\mathrm{Spec}}
\DeclareMathOperator{\Gal}{\mathrm{Gal}}

\DeclareMathOperator{\Var}{\mathrm{Var}}
\DeclareMathOperator{\Codim}{\mathrm{codim}}
\DeclareMathOperator{\Dim}{\mathrm{dim}}

\DeclareMathOperator{\Sch}{\mathrm{Sch}}

\DeclareMathOperator{\Id}{\mathrm{id}}

\newcommand{\LL}{\ensuremath{\mathbb{L}}}
\newcommand{\A}{\ensuremath{\mathbb{A}}}

\usepackage{xcolor}
\usepackage{lipsum}
\usepackage[draft]{todonotes}   

\usepackage[foot]{amsaddr}
\title{The quotient map on the equivariant Grothendieck ring of varieties}
\author{Annabelle Hartmann}
\address{    Universität Bonn\\
   Endenicher Allee 60\\
D-53115 Bonn\\
    Tel.: 0049-228-7362348}
    \email{ahartman@math.uni-bonn.de}

\begin{document}
\parindent 0em

\begin{abstract}
For a scheme $S$
with a good action of a finite abelian group $G$
having enough roots of unity
we show that the quotient map on the $G$-equivariant Grothendieck ring of varieties over $S$
is well defined
with image in the Grothendieck ring of varieties over $S/G$ in the tame case,
and in the modified Grothendieck ring in the wild case.
To prove this we use a result on the class of the quotient of a
vector space by a quasi-linear action in the Grothendieck ring of varieties
due to Esnault and Viehweg,
which we also generalize to the case 
of wild actions.
As an application we deduce that the quotient of the motivic nearby fiber
is a well defined invariant.\\\\
%
\end{abstract}

\maketitle

\section{Introduction}

\noindent The Grothendieck ring $K_0(\Var_S)$ of varieties over a separated scheme $S$
is as group spanned by isomorphism classes $[X]$ of separated schemes $X$ of finite type over $S$
with relations allowing to cut and paste.
The ring structure is given by the fiber product.
This ring is useful because additive invariants of varieties,
for example the Euler characteristic 
and the number of points over a finite field in positive characteristic,
factor through this ring.
Therefore the Grothendieck ring of varieties and
localizations of it are used in motivic integration
as universal value rings.

Let $S$ now be a scheme with a good action of a finite group $G$,
and consider the category $(\Sch_{S,G})$ of separated $S$-schemes with good $G$-action.
A group action on $S$ is called \emph{good} if every orbit lies in an affine subscheme of $S$,
which insures that the quotient exists in the category of schemes,  see Section~\ref{assumtions}. 
The equivariant Grothendieck ring $K_0^G(\Var_S)$ of varieties over $S$, see Definition \ref{equivariant Gring},
is generated by isomorphism classes $[X]$ of objects $X$ in this category.
Whenever $Y$ is a $G$-invariant closed subscheme of $X$,
one asks the class of $X$ to be equal to the sum of the class of $Y$
and the class of $X\setminus Y$.
Moreover one asks the classes of two affine bundles with affine $G$-action to be equal if they have the same rank and the same base.
The ring structure is again given by the fiber product.

Using the
equivariant Grothendieck ring of varieties 
as value ring allows us to also encode some group action on a scheme $X$,
as done for example with the monodromy action on the motivic Zeta function, see \cite{DL3}.
To get a well defined theory of motivic integration with group actions,
one needs to make actions on affine bundles 'trivial'.
This is where the last relation in the definition of the equivariant Grothendieck ring of varieties actually comes from.

One can ask now how to relate the equivariant Grothendieck ring with the usual one.
A natural thing to do is to divide 
out the action,
i.e., to send the class $[X]\in K_0^G(\Var_S)$ of a scheme $X$ with good $G$-action
to the class of its quotient $X/G$ in $K_0(\Var_{S/G})$.
Bittner showed that such a quotient map
is well defined if $S$ is a variety over a field of characteristic zero and the action of $G$ on $S$ is free,
see \cite[Lemma~3.2]{MR2106970}.
The proof uses
that in this case also the action on an affine bundle in the category $(\Sch_{S,G})$ is free,
and thus the quotient 
is again an affine bundle.
For general $G$-action on $S$ this is not the case.

\medskip

In this paper,
we show that for an abelian group $G$, the quotient map on the $G$-equivariant Grothendieck ring is well defined in general
if we put some extra assumptions on the stabilizers of the points of $S$.
 For $s\in S/G$, denote by $F_s$ the residue field of $s$,
 and by $G_s\subset G$ the stabilizer of a point $s'\in S$ in the inverse image of $s$ under the quotient map $S\to S/G$.
 With this notation, we show the following theorem:
\begin{thm*}[Theorem \ref{quotientmap}]
 Let $G$ be a finite abelian group.
 Assume for all $s\in S/G$ that
 $F_s$  contains all $\lvert G_s \rvert$-th roots of unity.
 Then there is a well defined group homomorphism
  \[
 K^G_0(\Var_S)\to K_0^*(\Var_{S/G})
 \]
 sending $[X]\in K_0^G(\Var_S)$ to $[X/G]\in K_0(\Var_{S/G})$ for every $X\in (\Sch_{S,G})$.
\end{thm*}
\noindent
Here $K_0^*(\Var_{S/G})=K_0(\Var_{S/G})$, if the $G$-actions on $S$ is \emph{tame}, 
i.e., if the characteristic of the residue field of every point in $S$ is prime to the order of its stabilizer, see Section \ref{assumtions}.
If the action of $G$ on $S$ is \emph{wild}, i.e., not tame,
$K_0^{*}(\Var_{S/G})$ is equal to the modified Grothendieck ring $K_0^{\text{mod}}(\Var_{S/G})$,
in which classes of varieties connected by universal homeomorphisms are equal, see Definition \ref{dfn modified}.
This is due to the fact that
if $G$ acts wildly on a scheme $X$
the quotient of a closed invariant subscheme of $X$ only has a universal 
homeomorphism onto its image in the quotient $X/G$,
which is in general not a piecewise
isomorphism on the underlying reduced schemes, see Example \ref{exinsep3}.
It is not known whether two such schemes have the same class in the usual Grothendieck ring of varieties.

\noindent
In order to prove that the quotient map is well defined,
we need to control in particular quotients of affine bundles by affine actions.
To do so, we show that the class of the quotient $V/G$ of an affine bundle $\varphi:V\to B$
in $K_0^*(\Var_{S/G})$
only depends on the rank $d$ of the bundle and its base $B$, see Lemma~\ref{lemma}.
We prove the lemma by showing that all fibers of the induced map $\varphi_G:V/G\to B/G$ 
have the class of an affine space of dimension $d$
in the Grothendieck ring of varieties,
and conclude then by spreading out.
To compute the fibers of $\varphi_G$
we use the following proposition:
\begin{prop*}[Proposition \ref{lemwildandtame}]
  Let $G$ be a finite abelian group with quotient $G\to \Gamma$.
  Let $k$ be a field of characteristic $p>0$,
  let $q$ be the greatest divisor of $\lvert G\rvert$ prime to $p$, 
  and let $K/k$ be a Galois extension with Galois group $\Gamma$.
  Assume that the Galois action on $K$ lifts to a $k$-linear action of $G$
  on a finite dimensional $K$-vector space $V$.
  If $k$ contains all $q$-th roots of unity, then
  \[
   [V/G]=\mathbb{L}_k^{\Dim_KV}\in K_0^{*}(Var_k)
  \]
  with $\mathbb{L}_k:=[\mathbb{A}_k^1]\in K_0^{*}(Var_k)$.
\end{prop*}
\noindent
This proposition was already shown in the tame case in
\cite[Lemma~1.1]{MR2642161}
by decomposing $V$ into eigenspaces.
This method does not work in the case of wild actions.
Instead we construct a $G$-equivariant map from $V$ to a vector space $W$ of dimension one over $K$,
use an induction argument to compute the fibers of the induce map between the quotients,
and use again spreading out to conclude.


\medskip

As an application of our main theorem, 
we get that 
the quotient of the motivic nearby fiber is a well defined invariant
with values in $\mathcal{M}_k$, the localization of $K_0(\Var_k)$ with respect to $\mathbb{L}:=[\mathbb{A}_k^1]$,
with $k$ a field of characteristic zero containing all roots of unity.
The motivic nearby fiber, see Definition \ref{dfn mnf},
is an invariant of  a non-constant 
morphism $f:X\to \mathbb{A}_k^1$,
with $X$ an irreducible smooth $k$-variety,
and was constructed in \cite{DL3}
as a limit of the motivic Zeta function.

We show moreover that modulo $\mathbb{L}$, the quotient of the motivic nearby fiber is equal to the motivic reduction $R(f)$ of $f$
in the image of $K_{0}(\Var_k)$ in $\mathcal{M}_{k}$,
see Proposition~\ref{application}.
The \emph{motivic reduction} of $f$, see Definition \ref{dfn mr}, 
is defined as the class of $h^ {-1}(X_0)$ in $K_0(\Var_{k})$ modulo $\mathbb{L}$,
where $h:Y\to X$ is any smooth modification of $f$, i.e., $Y$ is a smooth $k$-variety and $h$ is a proper morphism
inducing an isomorphism $Y\setminus h^ {-1}(X_0)\to X\setminus X_0$.
The definition of $R(f)$ does not depend on the choice of such an $h$ due to weak factorization, see Proposition \ref{welldef R(f)}.

From this result we deduce that,
if $X$ is a smooth variety with a proper, non-constant morphism $f:X\to \mathbb{A}_k^1$,
and the generic fiber $X_\eta:=X\times_{\A^1_k}\A^1_k\setminus \{0\}=X\setminus X_0$ of $f$ is equal to $1$ modulo $\mathbb{L}$
in $K_0(\Var_{\A^1_k\setminus \{0\}})$,
then the same holds for the special fiber $X_0$ of $f$ in the image of $K_0(\Var_k)$ in $\mathcal{M}_k$.
This can be seen as a motivic analog of the main theorem in \cite[Theorem~1.1]{MR2247971},
which says the following:
if $V$ is an absolutely irreducible smooth projective variety
over a local field $K$ with finite residue field $F$
which has a certain cohomological property,
namely that the \'etale cohomology of $V\times_K \bar{K}$ has coniveau~$1$,
then the amount of points of the special fiber of every projective regular model of $V$ is equal to $1$ modulo $\lvert F \rvert$.
We will explain this analogy in more details in Section~\ref{an application}.

\section{Preliminaries}
\label{assumtions}

\noindent
Fix a finite group $G$. 
Let $S$ be a separated scheme endowed with
a left action of $G$. 
If not mentioned otherwise, all group actions will be left actions.
We say the action of $G$ on $S$ is \emph{good} if every
orbit of this action is contained in an affine open subscheme of $S$.
By requiring the action to be good, one makes sure that the quotient exists in the category of schemes, see \cite[Expos\'e V.1]{MR0238860}.
We call an action \emph{tame} if the characteristic
of the residue field of every point $s\in S$
is zero or positive and prime to the order of the stabilizer $G_s\subset G$ of $s$. 
We call an action \emph{wild} if it is not tame.

If not mentioned otherwise,
we assume for the rest of the text that 
$S$ is a separated scheme with a good $G$-action,
that the quotient $S/G$ is locally Noetherian and separated, and that the quotient map $S\to S/G$ is finite.
This is for example true if $S$ is a separated scheme of finite type over a field $k$,
and $G$ acts on $S$ by a group of $k$-morphisms, see \cite[Expos\'e V.1, Corollaire 1.5]{MR0238860}.

We denote by $(\Sch_{S,G})$ the category whose objects are separated
schemes of finite type over $S$ with a good $G$-action such that the structure map is $G$-equivariant, and whose morphisms are
$G$-equivariant morphisms of $S$-schemes.
Note that if $G$ acts tamely on $S$, the same is true for every $X\in (\Sch_{S,G})$,
because the stabilizer of a point $x\in X$ is 
a subset of the stabilizer of the the image of $x$ in $S$.

One can check that the fiber product exists in this category:
take any $X,Y$ in $(\Sch_{S,G})$. 
For $g\in G$ let $g_X\in \Aut(X)$ and $g_Y\in \Aut(Y)$ be the corresponding automorphisms.
Then $g_X \otimes g_Y$ is an automorphism of $X\times_SY$.
Doing the same for every $g\in G$ we get an action of $G$ on $X\times_SY$ with $G$-equivariant projection maps.
This action is good, because the fiber product is constructed using affine covers.
It is easy to see that $X\times_SY$ together with the projection maps to $X$ and $Y$ is in fact the categorical fiber product in $(\Sch_{S,G})$.

\section{Equivariant affine bundles}
\label{equvariant affine bundles}

\begin{defn}\label{affine bundle}
Let $B$ be an $S$-scheme.
An \emph{affine bundle over $B$ of rank $d$} is a $B$-scheme $V$
with a vector bundle $E\to B$ of rank $d$ and a $B$-morphism $\varphi: E\times_BV\to V$ such that 
$\varphi\times p_V: E\times _BV \to V\times_B V$, where $p_V$ denotes the projection to $V$, is an isomorphism of $B$-schemes.
We call $E$ the \emph{translation space} of $V$.

An affine bundle $V$ over $B$ is called \emph{$G$-equivariant}, if $V$ and $B$ are in $(\Sch_{S,G})$, and $V\to B$ is $G$-equivariant.
The $G$-action on $V\to B$ is called \emph{affine} if there is a $G$-action on $E$, linear over the action on $B$,
such that $\varphi$ is $G$-equivariant.
An action on $E$ is \emph{linear over the action on $B$}
if for all $g\in G$ the
map $g':E \to g_B^*E$ induced by the following Cartesian diagram
\begin{equation*}\label{diagram1}
  \xymatrix{
 E\ar@/^/[drr]^{g_E}\ar@/_/[ddr]\ar@{-->}[rd]^{g'}\\
 &g_B^*(E)\ar[r]\ar[d]& E\ar[d]\\
 & B\ar[r]^{g_B}& B
 }
\end{equation*}
 is a morphism of vector bundles.
 Here $g_B\in \Aut(B)$ and $g_E\in \Aut(E)$ are the automorphisms of $B$ and $E$ induced by $g$,
 and $g_B^*E:=E\times_B B$, where $B$ is a $B$-scheme via $g_B$.
\end{defn}

\begin{ex}\label{exe=v}
 Let $E$ be a vector bundle over some $B\in (\Sch_{S,G})$ with an action on $G$ which is linear over the action on the base $B$,
 then $E$ can also be viewed as a $G$-equivariant affine bundle with affine $G$-action.
 We call the $G$-action on $E$ \emph{quasi-linear}.
\end{ex}

\begin{ex}\label{exwildtranslation}
Let $k$ be a field of characteristic $p>0$, and let $G=\mathbb{Z}/p\mathbb{Z}$.
Let $B=\Spec(k)$ with trivial $G$-action, and consider $V= \Spec(k[x])$ with the $G$-action given by sending $x$ to $x+1$.
As this action has no fixed point, there is no way of changing coordinates to achieve that this action is linear.
Hence in particular $V$ is not isomorphic to a $k$-vector space with linear action.

Let $E=\Spec(k[y])$ be the trivial vector bundle of dimension $1$ over $B$
with trivial action of $G$.
Consider the map
 given by sending $(e,v)\in E\times_BV$ to $e+v\in V$.
 As $(e,v+1)$ is mapped to $e+v+1$ this map is clearly $G$-equivariant.
One can check that it induces an isomorphism $V\times_BE\to V\times_B V$.
So $V\to B$ is a $G$-equivariant affine bundle
with affine $G$-action, because the action on $E$ is trivial.
 \end{ex}
 
 \begin{rem}\label{trivialtorsor}
Definition \ref{affine bundle} implies that a $G$-equivariant affine bundle $V\to B$ of rank $d$ is in particular a \emph{principal homogenous space} or \emph{torsor}.
 Hence \cite[Proposition 4.1]{MR559531} implies that it is locally in the \'etale topology a \emph{trivial torsor},
 i.e., there is a cover $\{U_i\}_{i\in I}$ of $B$ in the \'etale topology,
 such that $V_{U_i}:=V\times_BU_i\cong E_{U_i}:=E\times_B U_i$, and $E_{U_i}\cong \mathbb{A}^d_{U_i}$ acts on $V_{U_i}$ by translation.
  
  By \cite[Propostion 14]{Serre} the algebraic group $\mathbb{G}_a^d$ is special.
  This means by definition, see \cite[4.1]{Serre}, that a locally trivial $\mathbb{G}_a^d$-torsor in the \'etale topology is already locally trivial in the Zariski topology.
  But if we restrict $B$ to an open over which $E$ is trivial, $V$ is a locally trivial $\mathbb{G}_a^d$-torsor in the \'etale topology.
  Hence we may assume that $\{U_i\}_{i\in I}$ is a cover of $B$ in the Zariski topology. 
 \end{rem}

\begin{rem}\label{Vb} 
Let $B\in (\Sch_{S,G})$, and let $E$ be a vector bundle of rank $d$ with a $G$-action which is linear over that on $B$.
Let $b\in B$ be a \emph{fixed point}, i.e., the orbit of $b$ under the action of $G$ on $B$ contains only $b$,
and let $K$ be its residue field.
Then $E_b:=E\times_Bb\cong\Spec(K[x_1,\dots,x_d])$, and the $G$-action on $E$ restricts to $E_b$.
Take any $g\in G$, and let $\alpha\in \Aut(K[x_1,\dots,x_d])$ be the corresponding automorphism of rings.
Then we have the following commutative diagram.

\begin{equation*}\label{diagram affine}
  \xymatrix{
 K[x_1,\dots,x_d]\\
 &K\otimes_KK[x_1,\dots,x_d]\ar@{-->}[lu]^{\alpha'}& K[x_1,\dots,x_d]\ar[l]\ar@/_/[ull]^{\alpha}\\
 & K\ar[u]\ar@/^/[uul]& K\ar[l]^{\alpha\rvert_K}\ar[u]
 }
\end{equation*}
Note that $K\otimes_KK[x_1,\dots,x_d] \cong K[x_1,\dots,x_d]$, but the $K$-structure on the first
is given by sending $s\in K$ to $\alpha^{-1}(s)$.
By the definition $\alpha'$ is $K$-linear.
Hence we have
\begin{equation}\label{linear}
\alpha(x_i)=\alpha'(x_i)=\sum_{j=1}^d a_{ij}x_j
\end{equation}
for some $a_{ij}\in K$.
Using that $\alpha$ is a ring morphism,
we get that
\begin{equation}\label{qlin}
 \alpha(v+sw)=\alpha(v)+\alpha\rvert_K(s)\alpha(w)
\end{equation}
for all $v,w\in K[x_1,\dots,x_d]$ and $s\in K\subset K[x_1,\dots,x_d]$.
If $s\in k:=K^G$,
then $\alpha(v+sw)=\alpha(v)+s\alpha(w)$,
because $\alpha\rvert _k=\Id$ by definition.
Note that $K$ is a Galois extension of $k$,
and we have a surjective map 
\[
G\to \Gal(K,k)=:\Gamma.
\]
So $K$ is a $k$-vector space of dimension $r:=\lvert\Gamma\rvert$.
Now we can view $E_b$ as a $K$-vector space of dimension $d$, and
hence also as a $k$-vector space of dimension $rd$.
We have seen that the $G$-action on $E_b$ defines a $k$-linear action on $E_b$
which lifts the Galois action of $\Gamma$ on $K$.
This follows from Equation (\ref{linear}) and Equation (\ref{qlin}).
Hence in particular the $G$-action on $E_b$ is quasi-linear.
\end{rem}

\begin{rem}\label{Vb translation}
Let $V\to B$ be a  $G$-equivariant affine bundle of rank $d$ with affine $G$-action with translation space $E\to B$,
and let $b\in B$ be a fixed point.
 Remark~\ref{trivialtorsor}
 implies that $\varphi_b: E_b\times V_b \to V_b$, with $E_b=E\times_B b$ and $V_b=V\times_B b$,
 is the trivial torsor,
 hence $V_b\cong E_b \cong \mathbb{A}^d_K$,
 and $\varphi_b$ sends $(v,w)\in E_b\times V_b$ to $v+w\in V_b$.
 
 As $b$ is fixed under the action of $G$, the $G$-action on $E$ and $V$ restrict to $E_b$ and $V_b$.
 Moreover $\varphi_b$ is $G$-equivariant.
 Take any $g\in G$, and let $g_E\in \Aut(E_b)$ and $g_V\in \Aut(V_b)$ be the corresponding automorphisms.
 Fix a $0\in V_b$.
For all $v\in V_b$ we have that
 \begin{equation}\label{formel translation}
   g_V(v)=g_V(v+0)=g_V(\varphi_b ( v,0 ))=\varphi_b (g_E(v), g_V(0))=g_E(v)+g_V(0).
 \end{equation}
Note that Remark \ref{Vb} implies that
$g_E$ is quasi-linear.
Moreover $g_V(0)$ does only depend on $g$ and the choice of $0$, but not on $v$.
\end{rem}

\begin{rem} \label{tame translation}
 Assumption and notation as in Remark \ref{Vb translation}.
Let $H\subset G$ 
be the subgroup consisting of all elements of order prime to the characteristic of $K$.
Assume that $H$ is abelian.

View $V_b$ as a vector space over $k=K^G$, and
consider the action of $H$ on $V_b$.
By Remark~\ref{Vb translation} we know that
for every $h\in H$ the corresponding automorphism sends $v\in V_b$ to $A_h(v)+b_h$,
where $A_h$ is a $k$-linear map and $b_h\in V_b$.
We are now going to show that the action of $H$ on $V_b$ has a fixed point.
Therefore we view $V_b$ as a scheme over $k$,
hence $V_b\cong \mathbb{A}^{rd}_k=\Spec(k[x_1,\dots,x_{rd}])$,
and show by induction on $n:=\lvert H\rvert$ that
the fixed point locus $V_b^H\subset V_b$ is isomorphic to $\mathbb{A}_k^N$ for some $N\geq 0$.
For $n=1$, the statement is trivial.

So let $n>1$.
Then there exists a nontrivial cyclic $q$-subgroup $H'$ of $H$ for some prime $q$, prime to the characteristic of $k$.
Consider the induced action of $H'$ on $V_b$.
As $q\neq p$, we can use \cite[Corollary 5.5]{1009.1281}, which follows from a theorem of Serre in \cite{MR2555994},
to get that $V_b^{H'}(\bar{k})\neq \emptyset$.
Here $\bar{k}$ is the algebraic closure of $k$.
In particular $V_b^{H'}$ is not the empty scheme.
Let $h\in H$ be a generator of $H'$.
Then the corresponding automorphism of $k[x_1,\dots,x_{rd}]$ sends $x_i$
to $\sum h_{ij}x_j +h_i$ for some $h_{ij},h_i\in k$.
Hence $V_b^{H'}\subset V_b$ is given by equation of the form
$h_{ij}x_j + h_i-x_i$, hence $V_b^{H'}$ is a nonempty linear subspace of $V_b$,
so in particular isomorphic to $\mathbb{A}_k^N$ for some $N\geq 0$.

As $H$ is abelian, it maps every point fixed by $H'$
to a point fixed by  $H'$.
Hence $V_b^{H'}$ is $H$-invariant.
As $H'$ acts trivially on $V_b^{H'}$,
we get in fact an action of $H/H'$ on $V_b^{H'}$.
This action is still given by some $k$-linear maps composed with some translation.
As the order of $H/H'$ is smaller than $n$,
we can now use the induction assumption to get that
$V_b^H=(V_b^{H'})^{H/H'}\cong \mathbb{A}_k^N$ for some $N\geq 0$.
In particular $V_b^H$ has a point $v_0$ over $k$.

Let $g\in H\subset G$, and let $g_V$ be the corresponding automorphism of $V_b$.
Then $g_V(v_0)=v_0$.
If we now chose $0$ in Remark~\ref{Vb translation} to be $v_0$,
we get from Equation (\ref{formel translation}) that for all $g\in H\subset G$ we have
\begin{equation*}\label{eq tame translation}
 g_V(v)=g_E(v)+g_V(0)=g_E(v)+0=g_E(v)
\end{equation*}
for all $v\in V_b$.
\end{rem}

\noindent
Note that we can also find a fixed point in Remark \ref{tame translation} using elementary calculations
instead of \cite[Corollary 5.5]{1009.1281}.
In both cases
we need to assume that $H\subset G$ is an abelian subgroup.
Moreover it is crucial that
the order of $H$ is prime to the characteristic of $K$.
In case of wild actions there exist $G$-equivariant affine bundles with affine $G$-action
such that there is no change of coordinates making the action quasi-linear, even if $G$ is cyclic.
Such an $G$-equivariant affine bundle is given in Example~\ref{exwildtranslation}.

\section{The equivariant Grothendieck ring of varieties}
\label{equivariant Grothendieck ring}

\begin{defn}\label{equivariant Gring}
The \emph{equivariant Grothendieck ring of $S$-varieties}
$K_0^G(\Var_S)$ is defined as follows: as an abelian group, it is
generated by isomorphism classes $[X]$ of elements $X\in(\Sch_{S,G})$. These generators are subject to the following
relations:

\begin{enumerate}
\item $[X]=[Y]+[X\setminus Y]$, whenever $Y$ is a closed $G$-equivariant 
sub scheme of $X$ (scissors relation).
\item $[V]=[W]$, whenever $B\in (\Sch_{S,G})$, and
$V\rightarrow B$ and $W\to B$ are two
 $G$-equivariant affine bundles of rank $d$ over $B$ with affine $G$-action, see Definition \ref{affine bundle}.
\end{enumerate}
For all $X,Y\in (\Sch_{S,G})$, set 
 $[X]\cdot[Y]:=[X\times_S Y]$, where
 the fiber product is taken in $(\Sch_{S,G})$.
 This product extends bilinearly to
 $K_0^G(\Var_S)$ and makes it into a ring.
 
We denote by $\LL_S$
 the class of the affine line $\A^1_S$ with $G$-action induced by the action on $S$ as above.
  If the base scheme $S$ is clear from the
 context, we write $\LL$ instead of $\LL_S$.
 We define $\mathcal{M}^G_S$ as the localization
 $K_0^G(\Var_S)[\LL_S^{-1}]$.
  \end{defn}
 
 \begin{nota}
  If $G$ is the trivial group $\{e\}$, we write
 $K_0(\Var_S)$ and $\mathcal{M}_S$ instead of $K_0^G(\Var_S)$ and
 $\mathcal{M}^G_S$, receptively.
 Note that in this case Relation (2) becomes trivial.
 If $S=\Spec(A)$,
 we write $K_0^G(\Var_A)$ for $K_0^G(\Var_S)$, $\mathbb{L}_A$ for $\mathbb{L}_S$, and $\mathcal{M}_A^G$ for $\mathcal{M}_S^G$.
 \end{nota}

 \begin{rem} \label{rem 0div neg}
In \cite{lzero} it was shown that $\mathbb{L}_{\mathbb{C}}$ is a zero divisor in $K_0(\Var_{\mathbb{C}})$.
 This means in particular the canonical map $K_0^G(\Var_{S})\to \mathcal{M}_S^G$ is not an injective map in general.
 \end{rem}

\begin{rem}\label{remark}
A morphism of finite groups $G'\rightarrow G$ induces forgetful
ring morphisms
\[
K^{G}_0(\Var_S)\rightarrow K^{G'}_0(\Var_{S})\mbox{
and }
 \mathcal{M}^{G}_S\rightarrow \mathcal{M}^{G'}_{S}.
 \]
 If
 $G'\rightarrow G$ is surjective, then these morphisms are
 injections.
\end{rem}

\begin{defn}
 Let $S$ be a separated scheme with an action of a profinite group
 \[
 \widehat{G}=\lim_{\stackrel{\longleftarrow}{i\in I}} G_i
 \]
 factorizing through a good action of some finite quotient $G_i$.
 Then we define
 $$K_0^{\widehat{G}}(\Var_S):=\lim_{\stackrel{\longrightarrow}{i\in
 I}}K_0^{G_i}(\Var_S)\ \mathrm{and}\ \mathcal{M}^{\widehat{G}}_{S}:= \lim_{\stackrel{\longrightarrow}{i\in
 I}}\mathcal{M}^{G_i}_{S}.$$
 \end{defn}

 \medskip
 \noindent
 Note that in the literature on can find several different definitions of the equivariant Grothendieck ring of varieties.
 This difference always lies in relation~(2),
 which is needed to compute formulas in motivic integration.
 Our definition can be found for example in \cite[Section 3.4]{MR2483954}.
 We are now going to discuss two alternative definitions:
 
 \begin{rem}
 In \cite[2.9]{DLLefschetz} and \cite[2.4]{DL3}, instead of relation~(2) one divides out the following relation:
   \begin{itemize}
  \item[(2a)]
  $[V]=[W]$,
  where both $V$ and $W$ are affine spaces of degree $r$ over $B$ in $(\Sch_{S,G})$
  with good $G$-actions lifting the $G$-action on $B$.
 \end{itemize}
 The problem here is that it is very hard to say how these lifts look like. For example it is not even known
 whether or not all actions of a finite cyclic group on $\mathbb{A}_\mathbb{C}^3$ are linearizable, i.e., whether on can find coordinates for which the action becomes linear,
 see \cite[Section 6]{MR1423629}.
 In later definitions, there is always a restriction on the actions which one wants to consider.
\end{rem}
 
 \begin{rem}
 In \cite[Section 2.2]{MR2106970} one divides out the following relation instead of relation~(2):
 \begin{itemize}
  \item[(2b)] $[G\circlearrowright\mathbb{P}(V)]=[\mathbb{P}^n\times(G\circlearrowright B)]$, whenever
  $B\in (\Sch_{S,G})$ and $V\to B$
  is a vector bundle of rank $d+1$ with a $G$-action on $V$ which is linear over the action on $B$.
  Here $G\circlearrowright\mathbb{P}(V)$ denotes the projectivization of this action, whereas $\mathbb{P}^n\times(G\circlearrowright B)$
  denotes the action on $V$ on the right vector only.
  \end{itemize}
 Bittner uses this formulation, because she is working with projective varieties as generators for the Grothendieck ring,
see for example \cite[Corollary 3.6]{MR2106970}.
In this context it is of course important to work with projective varieties as generators for the relations. 
  
As already remarked in \cite{MR2106970},
relation~(2b) implies in particular that the class of two affine bundles of rank $d$ over $B$ with an affine action over the action on $B$
have the same class.

On the other hand,
let $V$ be a vector space of dimension $d+1$ over a field $k$ containing all $\lvert G \rvert$-th roots of unity and let $G$
be a finite abelian group acting linearly on $V$. These assumptions imply that
we can find a common eigenvector for the linear maps on $V$,
hence there are coordinates such that the induced action on $\mathbb{P}(V)$ sends
\[
 [x_0:\dots :x_d]\mapsto [x_0:\sum_{i=0}^da_{1i}x_i:\dots:\sum_{i=0}^da_{di}x_i]
\]
for all $g\in G$.
Hence we can decompose $\mathbb{P}(V)$ in the $G$-invariant subschemes given by $x_0=0$ and $x_0\neq 0$.
The first is of the the form $\mathbb{P}(V')$, where $V'$ is a $k$-vector space of dimension $d$ with a linear action and
the action on $\mathbb{P}(V')$ is induced by this action.
The second is isomorphic to $\mathbb{A}_k^d=\Spec(k[x_1,\dots,x_d])$
with affine $G$-action sending for every $g\in G$
\[
 (x_1,\dots ,x_d)\mapsto (\sum_{i=1}^da_{1i}x_i+a_{10},\dots,\sum_{i=1}^da_{di}x_i+a_{d0}).
\]
Using an inductive argument, we can decompose $\mathbb{P}(V)$ into $k$-vector spaces with affine $G$-action.
Analogously we can decompose $\mathbb{P}(V)$ if $V$ is a vector space over $K$ with a quasi-linear action of $G$ over a $G$-action on $K$.
Therefore we need to assume that $k:=K^G$ contains all $\lvert G\rvert$-th roots of unity.
With this decomposition one can use Proposition \ref{lemeh} to show as in Lemma \ref{lemma} (with the same assumptions on $S$ and $G$ as there)
that the class of the quotient of any $\mathbb{P}(V)$ as in relation~(2b)
only depends on the rank and the base of the vector bundle $V$, hence Theorem \ref{quotientmap} holds also for Bittner's definition.
To avoid the decomposition step it is more reasonable for us to work with our definition.
\end{rem}

 \section{The modified Grothendieck ring of varieties}
 \label{modified Grothendieck ring}
 
 \noindent
 Due to the nature of wild actions,
 we are not able to compute quotients of such actions in the usual Grothendieck ring
 by decomposing a scheme into $G$-invariant subschemes and computing the quotient separately on these subschemes.
 The quotient of a closed subscheme has in general
 a purely inseparable map to the image of this subscheme under the quotient map.
 But in the wild case this map
 might not be a piecewise isomorphism,
 as we will see in Example~\ref{exinsep3}.
 We do not know whether the classes of two schemes connected with such a
 morphism  have the same class in the Grothendieck ring of varieties.
Therefore we now introduce the modified Grothendieck ring of varieties, in which their classes are the same.
 
 \begin{defn}
A morphism of schemes $f:Y\to X$ is called a \emph{universal homeomorphism},
if for every morphism of schemes $X'\to X$ the morphism
of schemes $f': Y\times_X X'\to X'$ induced by base change is a homeomorphism.
 \end{defn}


 \begin{defn}\label{dfn modified}
 Let $\mathcal{I}_S\subset K_0(\Var_S)$ be the ideal
 generated by elements of the form $[X]-[Y]$
 such that there exists a universal homeomorphism
 $f:X\to Y$. The \emph{modified Grothendieck ring of $S$-varieties} is defined as the quotient
  \[
   K_0^{\text{mod}}(\Var_S):=K_0(\Var_S)/\mathcal{I}_S.
  \]
Denote by $\LL_S$
 the class of the affine line $\A^1_S$.
  If the base scheme $S$ is clear from the
 context, we write $\LL$ instead of $\LL_S$.
 We define $\mathcal{M}^{\text{mod}}_S$ as the localization
 $K_0^{\text{mod}}(\Var_S)[\LL_S^{-1}]$.
 \end{defn}
 
 \begin{nota}
  If $S=\Spec(A)$ is an affine scheme,
  we write $K_0^{\text{mod}}(\Var_A)$ for $K_0^{\text{mod}}(\Var_{S})$,
  $\mathbb{L}_A$ for $\mathbb{L}_S$,
  and $\mathcal{M}_A^{\text{mod}}$ for $\mathcal{M}_S^{\text{mod}}$.
 \end{nota}

  \begin{rem}\label{mod0}
  If $S$ is a Noetherian $\mathbb{Q}$-scheme, then the quotient map
  \[
   K_0(\Var_S)\to K_0^{\text{mod}}(\Var_S)
  \]
is an isomorphism, see \cite[Corollary 3.8.3]{MR2885336}.
In particular this holds if $S$ is a scheme of finite type over any field of characteristic $0$.

It is not known whether it is an isomorphism in positive characteristic.
The problem is that the standard specializing  morphisms used to distinguish elements in the Grothendieck ring factor through 
the modified Grothendieck ring, see \cite[Proposition 4.1]{k0mod}.
 \end{rem}

 \noindent
 We will now prove some technical lemmas which will be used later to compute
 quotients in the (modified) Grothendieck ring of varieties.
 
 \begin{lem}[Spreading out for the modified Grothendieck ring]\label{spreading out}
  Take a directed system of Noetherian commutative rings
  ${(A_i,\varphi_{ij}:A_i\to A_j)}$,
  and denote by $A$ the direct limit of this system in the category of rings.
  Then there exists an isomorphism of rings
  \[
\varphi^{\text{mod}}: \lim_{\stackrel{\longleftarrow}{i\in I}}K_0^{\text{mod}}(\Var_{A_i})\to K^{\text{mod}}_0(\Var_{A}).
  \]
 \end{lem}

 \begin{proof}
  Consider the ring morphism
    \[
\varphi:  \lim_{\stackrel{\longleftarrow}{i\in I}}K_0(\Var_{A_i})\to K_0(\Var_{A}).
  \]
  induced by the ring morphism $\varphi_i:K_0(\Var_{A_i})\to K_0(\Var_{A})$
  given by
  sending the class of an $A_i$-scheme $U$ to the class of $U\times_{\Spec(A_i)}\Spec(A)$.
   By \cite[Proposition~2.9]{MR2770561}, $\varphi$ is an isomorphism.
  As a universal homeomorphism is stable under base change,
  for every universal homeomorphism $f:X\to Y$ between two $A_i$-schemes,
  the base change of $f$ to $\Spec(A)$
 is also a universal homeomorphism.
  Hence we get well defined maps ${\varphi_i^{\text{mod}}:K_0^{\text{mod}}(\Var_{A_i})\to K_0^{\text{mod}}(\Var_A)}$,
  which induce a well defined surjective map $\varphi^{\text{mod}}$ as in the claim.
 
 We still need to show that $\varphi^{\text{mod}}$ is injective.
 So let $f: X\to Y$ be a universal homeomorphism between $A$-schemes.
  By \cite[Theorem 8.8.2]{MR0217086}
  there exist an $i$ and a morphism of $A_i$-schemes $f_i : X_i \to Y_i$
  such that the base change of $f_i$ to $\Spec(A)$ is $f$.
  By \cite[Theorem 8.10.5]{MR0217086} $f$ is a universal homeomorphism
  if and only if there is a $j\geq i$ 
  such that the base change of $f_i$ induced by ${\Spec(A_j)\to \Spec(A_i)}$ is a universal homeomorphism.
  Hence $\varphi^{\text{mod}}$ is injective.
  \end{proof}
 
 \noindent
Recall that we assume that $S$ is a separated scheme with good action
of a finite group $G$,
such that the quotient map $S\to S/G$ is finite
and $S/G$ is separated and locally Noetherian.
The next lemmas will enable us to decompose the quotient of schemes
in the category $(\Sch_{G,S})$ in the (modified) Grothendieck ring of varieties.

  \begin{lem}\label{closed mod} 
 Let $X\in (\Sch_{G,S})$, and denote
  by $\pi: X\to X/G$ the quotient.
   Let $Y\subset X$ be a closed $G$-invariant subscheme,
   and let $Z$ be the image of $Y$ under $\pi$.
   Then the $G$-action on $X$ restricts to a good $G$-action on $Y$,
   and there exists a universal homeomorphism 
   $f: Y/G\to Z$.
   Hence in particular
   \[
    [Y/G]=[Z]\in K_0^{\text{mod}}(\Var_{S/G}).
   \]
\end{lem}

  \begin{proof}
Let $i:Y\hookrightarrow X$ be the inclusion map.
As $Y\subset X$ is a $G$-invariant closed subscheme, 
the $G$-action on $X$ restricts to $Y$. 
As every affine subscheme of $X$
will restrict to an affine subscheme of $Y$,
this action is good.
Denote by $\pi_Y:Y\to Y/G$ the quotient map.
As $i$ is $G$-equivariant, we get an induced map $i_G:Y/G\to X/G$ with $\pi \circ i=i_G\circ \pi_Y$.
As $\pi$ maps $Y$ to $Z$, $i_G$ factors through $Z$.
We are going to show that ${i_G: Y/G \to Z}$ is a universal homeomorphism.
By \cite[2.4.5.]{MR0199181} it suffices to show that $i_G$ is finite, surjective and purely inseparable.

As both the points of $X/G$ and $Y/G$ are just orbits of the action of $G$,
the map $i_G: Y/G \to Z$ is a bijection on points.

As $X$ is of finite type over $S$ and hence over $S/G$ using that $S\to S/G$ is finite,
$\pi$ is finite
by \cite[Expos\'e V, Corollaire 1.5]{MR0238860}.
As $i$ is proper, the same holds for $\pi \circ i$.
As moreover $\pi_Y$ is surjective, $i_G$ is proper by \cite[Proposition 12.59]{MR2675155}.
We have already seen that $i_G$ is quasi-finite, hence it is finite.
It remains to show that $i_G$ is purely inseparable, i.e., that for all $y\in Y/G$
the residue field $L$ of $y$ is purely inseparable over the residue field $K$ of $z:=i_G(y)$.

Using \cite[Capitre V.2, Th\'eor\`eme 2]{MR782297}
we get the following:
let $G_x$ be the stabilizer of a point $x\in Y\subset X$ of the orbit of $G$ over $y$ and $z$, respectively,
and let $M$ be the residue field of $x$.
Then $M$ is normal over $L$ and over $K$, and $G_x$ surjects on $\Gal(M,L)$ and $\Gal(K,L)$.
Hence we get the following inclusions of fields
\[
 \xymatrix{
 L \ar@{^(->}[r] &M^{G_x}\ar@{^(->}[r]& M\\
 K.\ar@{^(->}[ru] \ar@{^(->}[u]&
 }
\]
As $M$ is normal over $K$,
 $M^{G_x}$ is normal over $K$, too.
We now can split this extension in a separable extension $K'$ over $K$,
and a purely inseparable extension $M^{G_x}$ over $K'$.
Observe that $K'$ is normal over $K$,
and therefore $K=K'^{\Gal(K',K)}$.
But $\Gal(K',K)$ is a quotient of $\Gal (M^{G_x},K)$, and the latter is trivial. Hence $K=K'$.
Therefore $M^{G_x}$ is purely inseparable over $K$, and hence the same holds for $L$.
This finishes the proof.
  \end{proof}

 \begin{lem}\label{rem tame decomposition}
 Assumptions and notation as in Lemma \ref{closed mod}.
 Assume moreover that the action of $G$ on $X$ is tame.
Then there exist a map $f:Y/G\to Z$ which is a piecewise isomorphism,
hence in particular
\[
    [Y/G]=[Z]\in K_0(\Var_{S/G}).
   \]
 \end{lem}
 
 \begin{proof}
Use the same notation as in the proof of Lemma \ref{closed mod}.
As $G$ acts tamely on $X$,
it follows from \cite[Capitre~V.2, Proposition 5 and Corollaire]{MR782297} that
 $M^{G_x}$ is actually equal to $L$ and to $K$, hence $L=K$.
 Hence we have a finite bijective map
 $i_G: Y/G \to Z$
 such that for every point $y\in Y/G$
 the residue field of $y$ is isomorphic with the residue field of  $i_G(y)$.
 As an $S/G$-variety has the same class in $K_0(\Var_{S/G})$ as the reduced underlying scheme,
 we can assume that both $Y/G$ and $Z$ are reduced.
  Take a generic point $\eta \in Y/G$ with function field $\kappa_\eta$.
  The image $i_G(\eta)\in Z$ will be a generic point with function field isomorphic to $\kappa_\eta$.
  Hence we can find open subschemes $U\subset Y/G$ and $V\subset Z$,
  such that $i_G:U\to V$ is an isomorphism.
  Now we can proceed with $i_G:Y/G\setminus U\to Z\setminus V$, and use Noetherian induction
  to get that $i_G$ is a piecewise isomorphism.
  The claim now follows using the scissors relation in $K_0(\Var_{S/G})$.
   \end{proof}

   \noindent
If the action of $G$ on $X$ is wild, we will really need
to work in the modified Grothendieck ring, as
the following examples show.

 \begin{ex}\label{exinsep}
 Let $p$ be a prime, $G=\mathbb{Z}/p\mathbb{Z}$ and let $k$ be a field of characteristic $p$.
  Let $X=\mathbb{A}^2_k=\Spec(k[x,y])$,
  and consider the action of $G$ on $X$ given by sending $x$ to $x+y$ and $y$ to $y$.
  We have that $X/G=\Spec(k[x^p+(p-1)xy^{p-1}, y])\cong \Spec(k[u,y])$.
  Denote by $\pi:X\to X/G$ the quotient map.
  
  Consider the $G$-invariant closed subscheme
  $Y=\Spec(k[x,y]/(y))\cong \A^1_k\subset X$.
Then the induced action on $Y$ is trivial,
 and the induced map $i_G$ from $Y=Y/G$ to $Z=\pi(Y)=\Spec(k[u,y]/(y))\cong \Spec(k[u])$ is given by sending $u$ to $x^p$.
 Note that $k(u)$ is the function field of $Z$
 and $k(x)$ is the function field of $Y/G=Y$.
 The field extension $k(u)\subset k(x)$ induced by $i_G$ is radical of degree $p$.
 As the characteristic of $k$ is equal to $p$, it is purely inseparable.
 
 Nevertheless, $Y$ and $Z$ are isomorphic over $k$, but this isomorphism is not given by $i_G$.
 \end{ex}

 \begin{ex}\label{exinsep3}
   Let $p>3$ be a prime, let  $G=\mathbb{Z}/p\mathbb{Z}$, and let $k$ be a field of characteristic $p$.
  Let $X=\mathbb{A}^6_k=\Spec(k[x,y,a,a',b,b'])$,
  and consider the action of $G$ on $X$ given by sending $P(x,y,a,a',b,b')\in k[x,y,a,a',b,b']$
  to ${P(x,y,a+a',a',b+b',b')}$.
  One can check that
  \begin{align*}
  X/G&=\Spec(k[x,y,a^p+(p-1)aa'^{p-1},a',b^p+(p-1)bb'^{p-1},b'])\\
  &\cong \Spec(k[x,y,u_a,a',u_b,b']).
 \intertext{ Denote by $\pi:X\to X/G$ the quotient map.
  Consider}
   Y&=\Spec(k[x,y,a,a',b,b']/(a',b',x^3+a^px-y^2+b^p))\\
  &=\Spec(k[x,y,a,b]/(x^3+a^px-y^2+b^p))\subset X.
  \end{align*}
  Note that $Y$ is $G$-invariant, and
 the induced action on $Y$ is trivial.
 As 
 \[
  x^3+u_ax-y^2+u_b=x^3+a^px-y^2+b^p+(p-1)axa'^{p-1}+(p-1)bb'^{p-1},
 \]
 we get that
 \[
 k[x,y,u_a,a',u_b,b']\cap (a',b',x^3+a^px-y^2+b^p)=(a',b',x^3+u_ax-y^2+u_b),
 \]
and hence
$Z=\pi(Y)=\Spec(k[x,y,u_a,u_b]/(x^3+u_ax-y^2+u_b))$.
Note that the residue field of the generic point of $Z$
is isomorphic to the function field of the elliptic curve
$E=\Spec(K[x,y]/(x^3+ax-y^2+b))$
with $K=K(a,b)\cong K(u_a,u_b)$.
Let $\varphi: K\to K$ be the Frobenius map.
Then 
\[
E^{(p)}:=E\times_{\varphi}\Spec(K)\cong \Spec(K[x,y]/(x^3+a^px-y^2+b^p)).
\]
Note that the residue field of the generic point of $Y=Y/G$
 is isomorphic to the function field of $E^{(p)}$.
 Now we compute the $j$-invariant of $E$  and $E^{(p)}$:
 \[
  j(E)=1728\frac{4a^3}{4a^3+27b^2} \text{ and } j(E^{(p)})=1728\frac{4(a^p)^3}{4(a^p)^3+27(b^p)^2} 
 \]
Hence $E$ and $E^{(p)}$ will have the same $j$-invariant
if and only if
 ${b^2}/{a^3}=({b^2}/{a^3})^p$.
As this was only true if ${b^2}/{a^3}$ was in $\mathbb{F}_p$,
the $j$-invariants
 of $E$ and $E^{(p)}$ are different.
 Therefore $E$ and $E^{(p)}$ are not isomorphic, and 
 hence $Z$ and $Y$ cannot be piecewise isomorphic.
 But by \cite[Theorem 2.13]{lzero} this does not imply that $Z$ an $Y$ cannot have the same class in the Grothendieck ring.
 In fact it is not known whether $E$ and $E^{(p)}$,
 and thus $Z$ and $Y$, have the same class in the Grothendieck ring.
 \end{ex}

 \begin{rem} \label{rem decomposing}
 Take $X\in (\Sch_{S,G})$, 
 let $Y\subset X$ be a closed $G$-invariant subscheme and let $U=X\setminus Y$ be the compliment.
 Let $\pi:X\to X/G$ be the quotient map.
 By
 \cite[Expos\'e V, Corollaire 1.5]{MR0238860}
 we have that
$U/G \cong \pi(U)$.
Hence
\begin{align*}
  [X/G]&=[U/G]+[Y/G]\in K_0(\Var_{S/G})\\
\shortintertext{if and only if}
 [Y/G]&=[\pi(Y)]\in K_0(\Var_{S/G}).
\end{align*}
We do not know whether this is true or not,
even if $S=\Spec(k)$ with $k$ a finite field, see Example \ref{exinsep3}.
 \end{rem}

\begin{lem} \label{lemma vezel}
Take two schemes $V,B\in (\Sch_{G,S})$, and let
${\varphi:V\to B}$ be a $G$-equivariant morphism of finite type.
 Denote the induced map between the quotients with ${\varphi_G:V/G\to B/G}$.
 Let $x\in B/G$ be a point with residue field $k$,
 let $b$ be a point in $B$ mapped to $x$ under the quotient map,
 and let $G_b$ be the stabilizer of $b$.
 Assume that $G_b$ is a normal subgroup of $G$.
 Then $G_b$ acts on $\varphi^{-1}(b)$, this action is good, and
 \[
  [\varphi_G^{-1}(x)]=[\varphi^{-1}(b)/G_b]\in K_0^{\text{mod}}(\Var_k).
 \]
\end{lem}

\begin{proof}
 Let $\pi:B\to B/G$ and $\pi_V:V\to V/G$ be the quotient maps.
 Note that $\pi\circ \varphi=\varphi_G\circ \pi_V$.
 Let $X\subset B/G$ be the closure of $x$ in $B/G$.
 By construction $\varphi^{-1}(\pi^{-1}(X))=\pi_V^{-1}(\varphi_G^{-1}(X))$
 is a $G$-invariant closed subscheme mapped surjectively to $\varphi_G^{-1}(X)$ under
 the quotient map $\pi_V$.
 Thus
 by Lemma \ref{closed mod}
there is a universal homeomorphism
\[
f: \varphi^{-1}(\pi^{-1}(X))/G\to \varphi_G^{-1}(X).
\]
Hence we get a universal homeomorphism
\[
f_k: \varphi^{-1}(\pi^{-1}(X))/G\times_{X} \Spec(k)\to \varphi_G^{-1}(X)\times_{X} \Spec(k)= \varphi_G^{-1}(x), 
\]
because universal homeomorphisms are stable under base change.
 As $x$ is the generic point of $X$,
 $\Spec(k)\to X$ is flat.
 Hence by \cite[Expos\'e V, Proposition~1.9]{MR0238860},
 \[
  \varphi^{-1}(\pi^{-1}(x))/G=(\varphi^{-1}(\pi^{-1}(X))\times_{X} \Spec(k))/G
  \cong \varphi^{-1}(\pi ^{-1}(X))/G\times _{X} \Spec(k).
 \]
Thus in the modified Grothendieck ring we get
\[
 [\varphi_G^{-1}(x)]=[\varphi^{-1}(\pi^{-1}(X))/G\times_{X}\Spec(k)] =[\varphi^{-1}(\pi^{-1}(x))/G]\in K_0^{\text{mod}}(\Var_k).
\]
Now consider the stabilizer $G_b$ of $b\in \pi^{-1}(x)$.
As $G_b$ is a subgroup of $G$, it acts on $V$ and $B$.
By construction, $\pi^{-1}(x)$ is $G$-invariant and $\varphi$ is $G$-equivariant,
hence we get induced actions of $G_b$ on $\varphi^{-1}(\pi^{-1}(x))$ and $\pi^{-1}(x)$,
and an induced map 
\[
\psi:\varphi^{-1}(\pi^{-1}(x))/G_b\to \pi^{-1}(x)/G_b.
\]
As $G_b\subset G$ is a normal subgroup, we may consider $H:=G/G_b$.
$H$ acts on $\varphi^{-1}(\pi^{-1}(x))/G_b$ and $\pi^{-1}(x)/G_b$,
and $\psi$ is $H$-equivariant.
As $\pi^{-1}(x)$ is the inverse image of the point $x$ of the finite quotient map $\pi$,
$\pi^{-1}(x)$ is a finite union of points.
Thus also $\pi^{-1}(x)/G_b$ is a finite union of points $P_1,\dots,P_n$,
and hence $\varphi^{-1}(\pi^{-1}(x))/G_b\cong \bigcup \psi^{-1}(P_i)$.
As $G_b$ is the stabilizer of $b$ and $\pi^{-1}(x)$ is the orbit of $b$,
the action of $H$ on $\pi^{-1}(x)/G_b$ is free and transitive,
i.e., for every pair $i,j$ there is a unique $h_{ij}\in H$ 
with $h_{ij}(P_i)=(P_j)$.
Hence $h_{ij}$ also maps $\psi^{-1}(P_i)$ isomorphically to $\psi^{-1}(P_j)$.

Let $W$ be the disjoint union of $n$ copies of $\psi^{-1}(P_1)$.
Let $H$ act on $W$ as follows:
for $h\in H$ with $h(P_i)=P_j$,
let the corresponding automorphism of $W$ map the $i$-th copy of $\psi^{-1}(P_1)$ identically to the $j$-th copy.
It is obvious that $W/H\cong \psi^{-1}(P_1)$.
Consider the map $\varphi: W\to \varphi^{-1}(\pi^{-1}(x))/G_b$ given on the $i$-th copy of $\psi^{-1}(P_1)$
by $h_{1i}\rvert_{\psi^{-1}(P_1)}$.
One can check that $\varphi$ is a $H$-equivariant isomorphism with $H$-equivariant inverse,
hence we get that
\[
 \varphi^{-1}(\pi^{-1}(x))/G=\varphi^{-1}(\pi^{-1}(x))/G_b/H\cong W/H\cong \psi^{-1}(P_1).
\]
Without loss of generality we may assume that $P_1$ is the image of $b$ under the quotient map.
As $b$ is a component of $\pi^{-1}(x)$, $\varphi^{-1}(b)$ is open in $\varphi^{-1}(\pi^{-1}(x))$.
Moreover  it is $G_b$-invariant, because $G_b$ is the stabilizer of $b$ and $\varphi$ is $G_b$-equivariant.
So by \cite[Expos\'e V, Corollaire 1.4.]{MR0238860} we get that
\[
 \psi^ {-1}(P_1)\cong \varphi^ {-1}(b)/G_b.
\]
All together we have 
\[
 [\varphi_G^{-1}(x)] =[\varphi^{-1}(\pi^{-1}(x))/G]=[\psi^{-1}(P_1)]=[\varphi^{-1}(b)/G_b]\in K_0^{\text{mod}}(\Var_k).
\]

\end{proof}

%
 
\begin{rem}\label{rem fiber tame}
Assumption and notation as in Lemma \ref{lemma vezel}.
If $G$ acts tamely on $S$, and hence also on $V$ and $B$,
 we get that 
 \[
  [\varphi_G^{-1}(x)]=[\varphi^{-1}(b)/G_b]\in K_0(\Var_k).
 \]
This holds, because in this case we get, as in the proof of Lemma \ref{rem tame decomposition}, that
$f$
is a finite, bijective map such that the map between the residue fields of the points are isomorphic.
 Hence the same holds for $f_k$,
 and we can show, as done in Lemma \ref{rem tame decomposition},
that $\varphi_G^{-1}(x)$
 and $\varphi^{-1}(\pi^{-1}(X))/G\times \Spec(k)\cong \varphi^{-1}(b)/G_b$
 have the same class in $K_0(\Var_k)$.
 \end{rem}

\section{Quotients of vector spaces by quasi-linear actions}
\label{quotients of vector spaces}
 
 \noindent
The aim of this section is to show a version of the following proposition in the case of wild group actions.
This proposition was proved in \cite[Lemma~1.1]{MR2642161}
as a generalization of \cite[Lemma 5.1]{Loo}.

 \begin{prop}\cite[Lemma~1.1]{MR2642161}\label{lemeh}
  Let $G$ be a finite abelian group with quotient $G\to \Gamma$.
  Let $k$ be a field of characteristic zero, or positive characteristic prime to $\lvert G\rvert$,
  and let $K/k$ be a Galois extension with Galois group $\Gamma$.
  Assume that the Galois action of $\Gamma$ on $K$ lifts to a $k$-linear action of $G$
  on a finite dimensional $K$-vector space $V$.
  If all $\lvert G\rvert$-th roots of unity lie in $k$, then
  \[
   [V/G]=\mathbb{L}_k^{\Dim_KV}\in K_0(Var_k).
  \]
 \end{prop}

 \noindent
 In \cite[Lemma~1.1]{MR2642161} this proposition was only stated for characteristic zero.
 Going through the proof, one recognizes
 that the only assumptions on $k$ which are used are that the characters of $G$
 are $k$-rational
 and that $\lvert G\rvert$ is prime to the characteristic of $k$,
 and hence for every representation of $G$ on a $k$-vector space $V$, there is a decomposition of $V$ into eigenspaces over $k$.
 This is not true if the characteristic of $k$ divides $\lvert G \rvert$,
 even if $k$ is algebraically closed.
Note furthermore that if the action of $G$ is tame,
we can decompose in the usual Grothendieck ring of varieties,
see Lemma~\ref{rem tame decomposition}.
Hence the proposition really holds in $K_0 (\Var_k)$, also if the characteristic of $k$ is positive but prime to the order of $G$.
 
 Working with wild actions,
 we cannot decompose $V$ into eigenspaces, so
 we will not be able to use the stratification of $V$ from \cite[Lemma~5.1]{Loo} as done in \cite[Lemma~1.1]{MR2642161}.
 Instead we will first show the claim for the case $\Dim_K V=1$ with elementary methods,
 and then use a $G$-equivariant fibration $\varphi: V\to W$ to a vector space $W$ of dimension $1$ over $K$
 to conclude by induction.
 More precisely,
 we compute the classes of the fibers of the induced map $\varphi_G: V/G\to W/G$ separately
 using the induction assumption and Lemma \ref{lemma vezel},
 and then we use spreading out, see Lemma~\ref{spreading out}.
 To be able to use Lemma \ref{lemma vezel} we need to work in the 
 modified Grothendieck ring of varieties.
 
\begin{prop}\label{lemwildandtame}
  Let $G$ be a finite abelian group with quotient $G\to \Gamma$.
  Let $k$ be a field of characteristic $p$,
  let $q$ be the greatest divisor of $\lvert G\rvert$ prime to $p$, 
  and let $K/k$ be a Galois extension with Galois group $\Gamma$.
  Assume that the Galois action on $K$ lifts to a $k$-linear action of $G$
  on a finite dimensional $K$-vector space $V$.
  If $k$ contains all $q$-th roots of unity, then
  \[
   [V/G]=\mathbb{L}_k^{\Dim_KV}\in K_0^{\text{mod}}(Var_k).
  \]
 \end{prop}

\begin{proof}
As $G$ is a finite abelian group, we have that
\[
 G\cong \mathbb{Z}/p^{r_1}\mathbb{Z}\times ... \times \mathbb{Z}/p^{r_s}\mathbb{Z}\times \mathbb{Z}/q_1\mathbb{Z}\times \dots \times \mathbb{Z}/q_t\mathbb{Z}
\]
with $p$ the characteristic of $k$, and $q_i$ prime to $p$.
Set $r:=\sum r_i$ and $q:=\prod q_i$.
Note that $\lvert G\rvert=p^rq$, and $q$ is prime to $p$.
 Set $d:=\Dim_KV$.
Then we have 
 $V\cong \Spec(K[x_1,\dots,x_d])$ as schemes.
  The $G$-action on $V$ is given by $\alpha_1,\dots, \alpha_s$ in $\Aut_k(K[x_1,\dots,x_d])$ with $\alpha_l^{p^{r_l}}=\Id$,
and $\beta_1,\dots, \beta_t$ in $\Aut_k(K[x_1,\dots,x_d])$ with $\beta_l^{q_l}=\Id$,
  such that
  $\alpha_l\rvert _K$ and $\beta_l\rvert_K$ generate the Galois action on $K$.
  As $G$ is abelian, the $\alpha_l$ and $\beta_l$ commute.
  
  View $V$ as a vector space over $k$. 
  The $G$-action is given by $A_l,B_l\in GL_k(V)$ by assumption.
  Moreover,
  $A_l ^{p^{r_l}}=\Id$ and $B_l^{q_l}=\Id$,
  and the $A_l$ and $B_l$ commute.
  Note that, as the order of the $B_l$ is finite of rang prime to $p$, the $B_l$ are diagonalizable over $\bar{k}$.
  As $B_l^{q_l}=\Id$, all eigenvalues are $q_l$-th roots of unity, and as $q_l$ divides $q$,
  all those eigenvalues are already in $k$ by assumption.
  Hence $B_l$ is already diagonalizable over $k$.
  As the $B_l$ commute, we find a basis of $V$ of common eigenvectors of all the $B_l$.
 
 Now consider the $A_l$.
 Let $E$ be any intersection of eigenspaces of the $B_l$.
 As the $A_l$ commute with the $B_l$,
 $A_l(E)=E$ for all $l$.
 Recall that $A_l^{p^{r_l}}=\Id$,
 hence all eigenvalues of $A_l$ are $p^{r_l}$-th roots of unity,
   and as $\Char(k)=p$, all the eigenvalues are $1$, i.e., in particular in $k$.
  So we find a $k$-basis of $E$
  such that $A_l$ has upper triangle form with only $1$ on the diagonal.
  As the $A_l$ commute, we can even find a $k$-basis of $E$ such that all the $A_l$ have upper triangle form.
  We can do this for all intersections of eigenspaces of the $B_l$, hence we 
 get a $k$-basis $B:=\{v_1,\dots, v_s\}$ of $V$ such that all $A_l$ have upper triangle form with only $1$ on the diagonal,
 and $B$ consist only of eigenvectors of the $B_l$.
 
  Consider the subset of $B$ containing
  those $v_i$ which do not lie in the 
  sub-$K$-vector space of $V$ spanned by the $v_j$ with $j<i$.
  This way we get a basis $B'=\{w_1,\dots,w_d\}$ of $V$ as $K$-vector space
  such that $A_l(w_i)=w_i + \sum_{j<i}a_{lij}w_i$ for some $a_{lij}\in K$,
  and $B_l(w_i)=\mu_{li}w_i$ for some $q_l$-th roots of unity $\mu_{li}$.
Hence we may assume - after a change of coordinates - that
  \[
  \alpha_l(x_i)=x_i+\sum_{j<i}a_{lij}x_j\text{ and }   \beta_l(x_i)=\mu_{li}x_i
  \]
  for some $a_{lij}\in K$, and
  some $q_l$-th roots of unity $\mu_{li}$.
  We need to show that
    \begin{equation}\label{claim}    
   [V/G]=[\Spec(K[x_1,\dots,x_d]^G)]=\mathbb{L}_k^d\in K_0^{\text{mod}}(Var_k).
    \end{equation}
We will show a slightly more general statement by induction on $d$,
we namely only ask that
$\alpha_l\rvert_K$ and $\beta_l\lvert_K$ generate the action of $\Gamma =\Gal(K,k)$,
and that
  \begin{equation}\label{translation}
  \alpha_l(x_i)=x_i+\sum_{j<i}a_{lij}x_j+a_{li} \text { and } \beta_l(x_i)=\mu_{li}x_i
  \end{equation}
  for some $a_{lij}\in K$ and $a_{li}\in K$, and some $q_l$-th roots of unity.
  
  \medskip
  
  Start with $d=1$. Hence $V\cong\Spec(K[x_1])$, $\alpha_l(x_1)=x_1+a_l$ for some $a_l\in K$,
  and $\beta_l(x_1)=\mu_lx_1$ for some $q_l$-th root of unity $\mu_l$.
  We will show the claim  for $d=1$ by jet another induction, this time on $r$.
  If $r=0$, $\lvert G\rvert$ is prime to $p$, so by \cite[Lemma~1.1]{MR2642161}  the claim is true.
  Assume hence that the claim holds for $r-1$. 
  
  Let ${G}':=\mathbb{Z}/p\mathbb{Z}\times \{0\} \times \dots \times \{0\}\subset G$.
  Note that $G'\subset G$ is a normal subgroup of order $p$ and acts on $V$.
  The action is generated by $\gamma=\alpha_1^{p^{r_1-1}},$
  and hence $\gamma(x_1)=x_1+b$ for some $b\in K$.
  Note that $K$ is a Galois extension of ${K}':=K^{{G}'}$ with Galois group generated by $\gamma\rvert_K$.
  
  If $b=0$, $K[x_1]^{G'}=K^{G'}[x_1]=K'[x_1]$.
  If $b\neq 0$ and $\alpha\rvert_K=\Id$, i.e., $K'=K$, then
  \[
  K[x_1]^{G'}= K[x_1 +(p-1)b^{1-p}x_1^p]\cong K'[y].
  \]
  We now may now assume that $K/K'$ is a field extension of degree $p$.
  As the characteristic of $K$ is equal to $p$,
  $K/K'$ is an Arthin-Schreier extension,
 thus $K=K'(\omega)$ for some $\omega$ in $K$,
  and $\gamma(\omega)=\omega+1$.
  Hence $\omega^2$ is mapped to $\omega^2 + 2\omega + 1$ and similarly for higher powers of $\omega$.
  As $\gamma$ is a linear map on the $K'$-vector space $K$ of degree $p$,
  we find a basis $\{1,\omega,v_3,\dots, v_{p} \}$ of $K$ over $K'$
  such that $\gamma(v_i)=v_i+v_{i-1}$ for $i>3$, and $\gamma(v_3)=v_3+\omega$.
  We can write $b$ in this basis, i.e., we have
  \begin{align*}
  b&=b_1 +b_2\omega + \sum_{i=3}^{p} b_iv_i\\
 \shortintertext
  {for some $b_i\in K'$.
  Set} 
  y&:=x_1 - b' \text{ with } b':=b_1 \omega + \sum _{i=2}^{p-1} b_i v_{i+1}.
  \end{align*}
  We have $\gamma(y)= y + b_{p}v_{p}$, and $K[x_1]\cong K[y]$.
  Using that the characteristic of $K$ is $p$, we get that $y =\gamma^p(y)=y+b_p$,
  hence $b_p=0$. 
  Therefore 
  \[
  K[x_1]^{G'}\cong K[y]^{G'}= K^{G'}[y]= K'[y].
  \]
  Now $H:=G/{G'}$ acts on $K[x_1]^{{G'}}$, which is isomorphic to ${K'}[y]$ for some $y$ as we have seen,
  and $(K[x_1]^{{G}'})^{H}=K[x_1]^G$.
  The action is given by $\alpha_l'=\alpha_l\rvert_{K'[y]}$,
  and $\beta_l'=\beta_l\rvert_{K'[y]}$.
  For simplicity we write $\alpha_l$ and $\beta_l$ also for $\alpha_l'$ and $\beta_l'$.
  Note that ${K}'$ is a Galois extension of $k$ and the Galois action is generated by $\alpha_l\rvert_{{K}'}$ and $\beta_l\rvert_{K'}$.
  In order to use the induction assumption, we still have to show
  that, maybe after some coordinate change, the $\alpha_l$ and $\beta_l$
  are given as in Equation (\ref{translation}).
  

If $y=x_1$, there is nothing to show.
Let $y=x_1 +\tilde{b}x_1^p$ with $\tilde{b}:=b^{1-p}(p-1)\in K$.
Recall that $\alpha_l(x_1)=x_1+a_l$ for some $a_l\in K$.
As $\alpha_l$ commutes with $\gamma$,
we get that $\alpha_l(b)+a_l=b+ \gamma(a_l)$.
As $\gamma\rvert_K=\Id$ in this particular case, it follows that $\alpha_l(b)=b$.
As $p-1\in k$,
we get that
$\alpha_l(\tilde{b})=\alpha_l(b)^{p-1}\alpha_l(p-1)=\tilde{b}$,
and hence
\[
 \alpha_l(y)=\alpha_l(x_1+\tilde{b}x_1 ^p)=x_1+\tilde{b}x_1^p+ a_l+\tilde{b}a_l^p=y+{a}_l'
\]
with ${a}_l'=a_l+\tilde{b}a_l^p\in K$.
Now recall that $\beta_l(x_1)=\mu_lx_1$ for some $q_l$-th root of unity $\mu_l$.
As $\beta_l$ and $\gamma$ commute,
$\beta_l(b)=\mu_lb$,
thus $\beta_l(\tilde{b})=\mu_l^{1-p}b^{1-p}(p-1)=\mu_l^{1-p}\tilde{b}$.
Hence
\[
 \beta_l(y)=\beta_l(x_1+\tilde{b}x_1^p)=\mu_lx_1+\mu_l^{1-p}\tilde{b}\mu_l^px_1^p=\mu_lx_1+\mu_l\tilde{b}x_1^p=\mu_ly.
\]
Consider now the case $y=x_1 -b'$ with $b'$ as above.
View $\beta_l$ and $\gamma=\alpha_1^{p^{r_1-1}}$ again as morphism of $K[x_1]$.
Recall that $\beta_l(x_1)=\mu_lx_1$ and $\gamma(x_1)=x_1+b$.
By assumption all the $\beta_l$ and $\gamma$ commute pairwise.
Hence
\[
 \mu_lx_1+\beta_l(b)=\beta_l(x_1+b)=\beta_l ( \gamma(x_1)) =\gamma( \beta_l(x_1))=\gamma (\mu_lx_1)=\mu_lx_1+\mu_lb,
\]
so $\beta_l(b)=\mu_lb$.
Note that the $\beta_l\rvert_K$ are $k$-linear maps of the $k$-vector space $K$.
As $\beta_l ^{q_l}=\Id$ and all the $q_l$-th roots of unity lie in $k$ by assumption,
there is a basis of eigenvectors of $\beta_l$ of $K$ over $k$.
As the $\beta_l$ commute with each other,
we can even find a common basis of eigenvectors of all the $\beta_l$.
Hence,
using that $b'\in K$,
we can write $b'=\sum b_{i}'$
with $b_i' \in K$, and
$\beta_l(b_{i}')=\mu_{li}b_{i}'$ for some $q_l$-th root of unity $\mu_{li}$.
Moreover we assume that if for all $l$ we have that $\mu_{li}=\mu_{lj}$, then $i=j$.
Without loss of generality we may assume that 
$\mu_{l1}=\mu_l$ for all $l$.
As $\gamma(K)\subset K$, and $b_i'\in K$,
$\gamma(b_{i}')=b_{i}'+\bar{b}_{i}$ for some $\bar{b}_{i}\in K$.
Using again that the $\beta_l$ and $\gamma$ commute,
we get that $\mu_{li}b_{i}' +\mu_{li}\bar{b}_{i}=\mu_{li}\bar{b}_{i}'+\beta_l(\bar{b}_{i})$,
i.e., $\beta_l(\bar{b}_{i})=\mu_{li}\bar{b}_{i}$ for all $l$.
In particular the $b_i$ which are not zero are linear independent.
As $\gamma (b')=b' +b$, we get that $b=\sum \bar{b}_i$.
Hence for all $l$ we have
\[
0 = \beta_l(b)-\mu_lb=\sum \beta_l(\bar{b}_i) -\sum \mu_{l}\bar{b}_{i}= \sum  (\mu_{li}\bar{b}_{i}-\mu_l \bar{b}_i)=\sum(\mu_{li}-\mu_{l})\bar{b}_i
\]
Hence $\bar{b}_{i}=0$ if $\mu_{li}\neq \mu_l$ for at least one $l$.
In particular this implies that $\tilde{b}:=\sum_{i\neq 1}b_i'$
lies in $K ^{G'}=K'$.
Set $\tilde{y}:=y+\tilde{b}=x_1-b_1'$.
It follows that
$K'[y]=K'[\tilde{y}]$, and
that $\beta_l(\tilde{y})=\beta_l(x_1-b_1')=\mu_lx_1-\mu_lb_1'=\mu_l\tilde{y}$.
Moreover $\alpha_l(\tilde{y})=\tilde{y}+(a_l - \alpha_l(b_1')+b_1')$,
and $a_l':=a_l - \alpha_l(b_1')+b_1'\in K'$.

 So all together we may assume that
 $K[x_1]^{G'}=K'[y]$ and
 $\alpha_l(y)=y+a_l'$ for some $a_l'\in K'$,
 and $\beta_l(y)=\mu_ly$.
 Hence we can use the induction assumption for the $H$-action on $K'[y]$.
This proves Equation~(\ref{claim}) for $d=1$.
 
 \medskip
 
 Now assume that the claim holds for $d-1$.
 Look at $V=\Spec(K[x_1,\dots,x_d])$ with a $G$-action as in Equation (\ref{translation}).
 Note that the inclusion map $K[x_1]\hookrightarrow K[x_1,\dots, x_d]$ is $G$-equivariant,
 if $G$ acts on $K[x_1]$ generated by $\alpha_l'$ and $\beta_l' $ such that the $\alpha_l'\rvert_K$ and $\beta_l' \rvert_K$ generate the Galois action on $K$,
 and $\alpha_l'(x_1)=x_1+a_{l1}$ and $\beta_l'(x_1)=\mu_{l1}x_1$.
 To simplify notation we will use $\alpha_l$ and $\beta_l$ also for $\alpha_l'$ and $\beta_l'$.
 Set $W:=\Spec(K[x_1])$,
denote by $\varphi$ the $G$-equivariant map from $V$ to $W$,
and let $\varphi_G: V/G\to W/G$ be the induced map between the quotients.

Let $x \in W/G$ be any point with residue field $\kappa_x$, and
let $w\in W$ be a point with residue field $\kappa_{w}$ in the inverse image of $x$ under the quotient map ${\pi: W\to W/G}$.
Let $G_w\subset G$ be the stabilizer of $w$.
We get an induced action of $G_w$ on $V$ and $W$.
Note that the action of $G_w$ on $V$ is generated by 
$\tilde{\alpha}_l =\alpha_l^{s_l}$ and $\tilde{\beta}_l=\beta_l^ {t_l}$ for some $s_l>0$ and some $t_l>0$, hence we have
  \[
  \tilde{\alpha}_l(x_i)=x_i+\sum_{j<i}\tilde{a}_{lij}x_j+\tilde{a}_{li} \text{ and } \tilde{\beta}_l(x_i)=\tilde{\mu}_{li}x_i
  \]
  for some $\tilde{a}_{lij}\in K$ and $\tilde{a}_{li}\in K$, and
with $\tilde{\mu}_{li}=\mu_{li}^ {t_i}$.

By construction, $w$ is fixed under this action of $G_w$,
and therefore we have that $\varphi^{-1}(w)\cong \Spec(\kappa_{w}[x_2,\dots,x_d])\subset V$
is $G_w$-invariant.
The action is given by $\gamma_l,\delta_l\in \Aut_k(\kappa_{w}[x_2,\dots,x_d])$ with
  \[
   \gamma_l(x_i)=x_i+\sum_{j<i, i\neq 1}\tilde{a}_{lij}x_j + (\tilde{a}_{li1}\bar{x}_1 + \tilde{a}_{li})\text{ and }
  \delta_l(x_i)=\tilde{\mu}_{li}x_i .
  \]
  Here $\bar{x}_1$ denotes the image of $x_1$ in $\kappa_{w}$.
  Note that $\tilde{a}_{li1}\bar{x}_1 + \tilde{a}_{li}\in \kappa_{w}$.
  Moreover $\kappa_{w}$ is a Galois extension of $\tilde{\kappa}_x :=\kappa_{w}^{G_w}$,
  and $\gamma_l\rvert_{\kappa_{w}}$ and $\delta_l\rvert_{\kappa_{w}}$ generate the Galois action.
  All together we can use the induction assumption and get that
  \[
  [\varphi^{-1}(w)/G_w]=[\Spec(\kappa_{w}[x_2,\dots,x_d]^{G_w})]= \mathbb{L}^{d-1}_{\tilde{\kappa}_x}\in K_0^{\text{mod}}(\Var_{\tilde{\kappa}_x}).
  \]
  As $\tilde{\kappa}_x$ is purely inseparable over $\kappa_{x}$ by \cite[Capitre V.2, Th\'eor\`eme 2]{MR782297}, and hence there is a universal homeomorphism $f: \mathbb{A}^{d-1}_{\tilde{\kappa}_x}\to \mathbb{A}^{d-1}_{\kappa_x}$,
  we get that
$\mathbb{L}^{d-1}_{\tilde{\kappa}_x}=\mathbb{L}^{d-1}_{\kappa_x}$ in $K_0^{\text{mod}}(\Var_{\kappa_x})$.
As $G$ is abelian and thus $G_w \subset G$ is a normal subgroup, we can use Lemma \ref{lemma vezel} to get that
 \begin{align}\label{punten}
[\varphi_G^{-1}(x)]=[\varphi^{-1}(w)/G_w]=\mathbb{L}_{\kappa_x}^{d-1}\in K_0^{\text{mod}}(\Var_{\kappa_\eta}).
 \end{align}
Now let $\eta\in W/G$ be the generic point with residue field $\kappa_\eta$.
By Lemma \ref{spreading out} there is an isomorphism
\[
\lim_{\stackrel{\longleftarrow}{\kappa_\eta \in U \subset W/G}}\!\!\!\!
K_0^{\text{mod}}(\Var_U)\to K_0^{\text{mod}}(\Var_{\kappa_\eta}).
\]
Hence using Equation (\ref{punten}) for $\eta$ there is a nonempty open $U\subset W/G$ such that
\[
[ \varphi_G^ {-1}(U)]=\mathbb{L}^ {d-1}_{U}\in K_0^{\text{mod}}(\Var_{U}).
\]
As the separated map $U\to \Spec(k)$ induces a forgetful map from $K_0^{\text{mod}}(\Var_U)$ to $K_0^{\text{mod}}(\Var_{k})$,
we have that $[\varphi_g^ {-1}(U)]=\mathbb{L}^ {d-1}_{k}[U]\in K_0^{\text{mod}}(\Var_{k})$.
Note that $W/G\setminus U$ consist of finitely many points $P_i$.
We already know that $[\varphi_G^{-1}(P_i)]=\mathbb{L}^d_k[P_i]$ in $K_0^{\text{mod}}(\Var_k)$,
see Equation (\ref{punten}).
Using the scissors relation in $K_0^{\text{mod}}(\Var_{S/G})$,
we get that 
\begin{align*}
V/G& =[\varphi_G^{-1}(W/G)]=[\varphi_G^{-1}(U)]+\sum [\varphi_G^{-1}(P_i)]\\
& =\mathbb{L}_k^{d-1}[U]+\sum \mathbb{L}_k^{d-1}[P_i]
=\mathbb{L}^ {d-1}_{k}[W/G]=\mathbb{L}_k^{d-1}
\mathbb{L}^{}_k=\mathbb{L}^d_k\in K_0^{\text{mod}}(\Var_{k}). 
\end{align*}
Here we used the induction assumption again to get that $[W/G]=\mathbb{L}_k\in K_0^{\text{mod}}(\Var_k)$.
 Claim (\ref{claim}) now follows by induction.
 \end{proof}
 
 \begin{rem}
  Consider $V=X=\Spec(k[x,y])$ from Example \ref{exinsep}.
  Then we get a $G$-equivariant map $\varphi: V\to W=\Spec(k[y])$,
  where we consider the trivial action on $W$,
  and hence a map $\varphi_g: V/G\to W/G=W$.
  We have already shown that the induced map between
  \[
  \varphi^{-1}(0)/G=\Spec(k[x])\to \varphi_G^{-1}(0)=\Spec(k[x^p+(p-1)xy^{p-1}])
  \]
  is purely inseparable,
  but $\varphi^{-1}(0)/G$ and $ \varphi_G^{-1}(0)$ are isomorphic over $k$,
  hence they have the same class in $K_0(\Var_k)$.
  This suggests that Proposition \ref{lemwildandtame} maybe also holds in the usual Grothendieck ring,
  at least if we assume that $k$ is a finite field.
  But as the isomorphism between the fibers is not given in a canonical way,
  I do not know how to do this in general, and whether it is possible at all.
  In the case that $\Dim_KV=1$, it follows from the proof of Proposition \ref{lemwildandtame}.
 \end{rem}
 
 \begin{rem}\label{exab}
 In the proof of Proposition \ref{lemwildandtame} we used the fact that $G$ is abelian
 to get a $G$-invariant sub-vector space $W$ of $V$.
 In fact this assumption is necessary.
As was already observed in \cite{MR2642161},
it follows from \cite[Proposition 3.1, ii]{Eke} and \cite[Corollary 5.2]{Eke}
that there exists an $n$-dimensional $\mathbb{C}$-vector space $V$ and a finite group $G\subset GL(V)$
such that 
\[
\lim_{m\to\infty}\ \ [V^m/G]/\mathbb{L}^{mn}\neq 1 \in \hat{\mathcal{M}}_{\mathbb{C}},
\]
where $\hat{\mathcal{M}}_{\mathbb{C}}$ is the completion of $\mathcal{M}_{\mathbb{C}}$ by the dimension fibration.
So in particular there exists an $m$ large enough such that $\mathbb{L}^{mn}\neq [V^ m/G]\in K_0(\Var_{\mathbb{C}})$.

In fact there are explicit $V$ and $G$ for which this holds, which can be found in \cite{MR751131}. 
The group in Saltman's example is a $p$-group, hence in particular solvable.
Hence Proposition \ref{lemeh} will not hold in the case of solvable groups. 
 \end{rem}
 
 \begin{rem}\label{exHE} 
 If one does not assume that $k$ has enough roots of unity,
  Proposition~\ref{lemeh} does not hold.
 There is namely
 a concrete counterexample, see \cite[Example~1.2]{MR2642161},
 of a $K=\mathbb{Q}(\sqrt{-1})$-vector space $V$ with $\mathbb{Q}$-linear
 $\mathbb{Z}/4\mathbb{Z}$-action lifting the action of the Galois group of $K$ over $\mathbb{Q}$,
 such that $[V/G]$ is not equal to $\mathbb{L}^{\Dim_KV}$ in
 $K_0(\Var_{\mathbb{Q}})$.
 \end{rem}

 \section{Quotients of equivariant affine bundles}
 \label{quotients of affine bundles}
 \noindent
Now we use the result from the previous section to compute the class of the quotient of an
equivariant affine bundle by an affine action
in a (modified) Grothendieck ring of varieties.
Take $S$ with an action of $G$ as before.
For a point $s\in S/G$,
denote by $F_s$ its residue field,
and by $G_s\subset G$ the stabilizer of a
point $s'\in S$
lying in the inverse image of $s$ under the quotient map $S\to S/G$.
A priory, $G_s$ depends on the choice of $s'$.
But as all stabilizers of an orbit are conjugated,
$\lvert G_s\rvert$ does not depend on $s'$.
 
 \begin{lem}\label{lemma}
 Let $G$ be a finite abelian group,
and let $\varphi:V\rightarrow B$ be
a $G$-equivariant affine bundle of rank $d$ with affine $G$-action in the category $(\Sch_{S,G})$.
Assume that the residue field $F_s$ of every point $s \in  S/G$
contains all $\lvert G_s \rvert$-th roots of unity.
Then
\[
[V/G]=\mathbb{L}_{S/G}^d[B/G]\in K_0^{\text{mod}}(\Var_{S/G}).
\]
If the action of $G$ on $S$ is tame, we get
\[
[V/G]=\mathbb{L}_{S/G}^d[B/G]\in K_0(\Var_{S/G}).
\]
 \end{lem}
 
 \begin{proof}
Let $\varphi:V\to B$  be as in the claim, and let $\varphi_G:V/G\to B/G$ be the induced map between the quotients.
 Let $x\in B/G$ be a point with residue field $k$.
 As $B/G$ is an $S/G$-scheme, $x$ lies over a point $s\in S/G$,
 and $k$ contains the residue field
$F_s$ of $s$.
 Let $b\in B$ be a point mapped to $x$ under the quotient map,
 let $K$ be the residue field of $b$ and let $G_b$ be the stabilizer of $b$ under the action of $G$.
 Without loss of generality we may assume that $b$ is mapped to $s'$ under
the structure map $B\to S$.
Hence as this map is $G$-equivariant,
$G_s$ is a subgroup of $G_b$.
As $G$ is abelian, $G_b\subset G$ is a normal subgroup, and hence
 Lemma~\ref{lemma vezel} implies that
\[
 [\varphi_G^{-1}(x)]=[\varphi^{-1}(b)/G_b]\in K_0^{\text{mod}}(\Var_k).
\]
If $G$ acts tamely on $S$,
we get this equation also in $K_0(\Var_k)$, see Remark \ref{rem fiber tame}.

As $\varphi:V\to B$ is a $G$-equivariant affine bundle of rank $d$ with affine $G$-action,
and $G_b$ is a subgroup of $G$,
the induced action of $G_b$ makes it into a $G_b$-equivariant affine bundle of rank $d$ with affine $G_b$-action.
As $G_b$ is the stabilizer of $b$, $b$ is fixed under the action of $G_b$.
Hence
by Remark \ref{Vb translation},
$V_b:=\varphi^{-1}(b)$ is a $K$-vector space of rank $d$,
and the $G_b$-action on $V$ restricts to $V_b$.
If the characteristic of $k$ and hence of $F_s$ is $p$,
this action is generated by automorphisms
$A_l,B_l\in \Aut(V_b)$ such that the order of the $A_l$ is a power of $p$
and the order of the $B_l$ divides $\lvert G_b\rvert$ and is prime to $p$.
As $G_b$ is a subgroup of $G_s$,
this means in particular that the order of $B_l$ divides $\lvert G_s\rvert$.
If the characteristic of $k$ is zero,
we may assume that the action is only generated by $B_l$s as above.
Also from Remark~\ref{Vb translation}
we get that for all $v\in V_b$ we have that
$A_l(v)=A_l'(v)+ u_l$  and $B_l(v)=B_l'(v)+w_l$ for some $u_l,w_l \in V_b$, 
such that $A_l',B_l'$ are quasi-linear maps of $V_b$.
As $G_b\subset G$ is abelian, we may assume by Remark \ref{tame translation} that $w_l=0$.

Note that $\tilde{k}:=K^{G_b}$
is a field extension of $k$ and hence of $F_s$.
By Remark \ref{Vb},
the $A_l'$ and the $B_l'$ define a $\tilde{k}$-linear action on a $K$-vector space
$V_b$
lifting the Galois action of $\Gal(K,\tilde{k})$ on K.
By assumption $F_s$ contains all $\lvert G_s\rvert$-th roots of unity, hence the same holds for $\tilde{k}$.
So using this and that $G_b$ is abelian, we can find as in the proof of Proposition \ref{lemwildandtame} a $K$-basis
such that the $A_l'$ have upper triangle form with only $1$ on the diagonal,
and the $B_l'$ are diagonal with $q_s$-th roots of unity as eigenvalues.
All together we may assume that $V_b\cong \Spec(K[x_1, \dots, x_d])$ as schemes, and that
the $G$-action on $V_b$ is given by
$\alpha_l,\beta_l \in \Aut(K[x_1,\dots,x_d])$ 
such that the $\alpha_l\rvert_K$ and $\beta_l\rvert_K$ generate the $\Gal(K,\tilde{k})$-action on $K$, and
\[
  \alpha_l(x_i)=x_i+\sum_{j<i}a_{lij}x_j+a_{li} \text { and } \beta_l(x_i)=\mu_{li}x_i
\]
for some $a_{lij},a_{l_i}\in K $ and some $q$-th roots of unity $\mu_{li}$.
But this is exactly the setting
for which we showed Proposition \ref{lemwildandtame}, see Equation (\ref{translation}).
Hence we get that
\begin{equation*}
 [\varphi^{-1}(b)/G_b]=[V_b/G_b]=\mathbb{L}_{\tilde{k}}^d\in K_0^{\text{mod}}(\Var_{\tilde{k}}).
\end{equation*}
As $G_b$ is the stabilizer of $b\in B$ under the action of $G$,
and $x$ is the image of $b$ under $\pi$,
by \cite[Capitre V.2, Th\'eor\`eme 2]{MR782297}
$\tilde{k}$ is a purely inseparable extension of $k$.
Hence
$\mathbb{L}_{\tilde{k}}=\mathbb{L}_k\in K_0^{\text{mod}}(\Var_k)$.
Putting everything together we get for every point $x\in B/G$ with residue field $k$ that
\begin{align}\label{punten2}
[\varphi_G^ {-1}(x)]=\mathbb{L}_k^d\in K_0^{\text{mod}}(\Var_{k}).
\end{align}
If the action of $G$ on $S$ is tame, we have that the characteristic of $F_s$ is prime to the order of $G_s$, and hence the characteristic of $k$
is prime to the order of $G_b$.
Thus by \cite[Capitre V.2, Proposition 5 and Corollaire]{MR782297}
$\tilde{k}=k$.
As the $A_l$ are trivial in this case, the $G$-action on $V_b$ is actually quasi-linear,
and hence by Proposition~\ref{lemeh}
we get that
\begin{align} \label{punten2a}
[\varphi_G^ {-1}(x)]=[\varphi^{-1}(b)/G_b]=\mathbb{L}_k^d\in K_0(\Var_{k}).
\end{align}
Note that, without loss of generality, we may assume that $B/G$ is integral in the claim.
This is true,
because we can decompose $B/G$ in irreducible schemes using the scissors relation.
This relation also implies that we can consider the underlying reduced scheme structure.
Let $\eta\in B/G$ be the generic point with residue field $\kappa_\eta$.
By Lemma \ref{spreading out} there is an isomorphism
\[
\lim_{\stackrel{\longleftarrow}{\kappa_\eta \in U \subset B/G}}\!\!\!\!
K_0^{\text{mod}}(\Var_U)\to K_0^{\text{mod}}(\Var_{\kappa_\eta}).
\]
This implies using Equation (\ref{punten2}) for $\eta$ that we can find an open $U\subset B/G$ such that
\[
[ \varphi_G^ {-1}(U)]=\mathbb{L}^ d_{U}\in K_0^{\text{mod}}(\Var_{U}).
\]
As the separated map $U\to S/G$ induces a forgetful map from $K_0^{\text{mod}}(\Var_U)$ to $K_0^{\text{mod}}(\Var_{S/G})$,
we get that $[\varphi_G^ {-1}(U)]=\mathbb{L}^ d_{S/G}[U]\in K_0^{\text{mod}}(\Var_{S/G})$.
Proceed with a generic point of $B/G\setminus U$.
By Noetherian induction we get, using the scissors relation in the modified Grothendieck ring,
that 
\[
[V/G]=\mathbb{L}^ d_{S/G}[{B/G}]\in K_0^{\text{mod}}(\Var_{S/G}).
\]
If the action of $G$ on $S$ is tame,
we can use that by \cite[Proposition 2.9]{MR2770561}
\[
\lim_{\stackrel{\longleftarrow}{\kappa_\eta \in U \subset B/G}}\!\!\!\!
K_0(\Var_U)\to K_0(\Var_{\kappa_\eta}).
\]
is an isomorphism, and Equation (\ref{punten2a})
to conclude with an analog argument as above that
\[
[V/G]=\mathbb{L}^ d_{S/G}[{B/G}]\in K_0(\Var_{S/G}).
\]
\end{proof}

\begin{ex}\label{extrivial}
 Take a $B\in (\Sch_{S,G})$,
 and let $\mathbb{A}^d_B\to B$ be
the trivial vector bundle with $G$-action induced  by the action on $B$.

As $\mathbb{A}_{S/G}^d$ is flat over $S/G$,
it also follows directly from  \cite[Expos\'e V, Proposition~1.9]{MR0238860} that
$\mathbb{A}^d_B/G\cong (\A^d_{S/G}\times_{S/G}B)/G \cong \A^d_{S/G}\times_{S/G} B/G$,
thus
\[
[\mathbb{A}^d_B/G]=\mathbb{L}^d_{S/G}[{B/G}]\in K_0(\Var_{S/G}).
\]
\end{ex}

\section{The quotient map on the equivariant Grothendieck ring of varieties}
\label{quotient map}

\noindent
In \cite[Lemma 3.2]{MR2106970}
it was shown that, if $G$ acts freely on a
variety $S$ of characteristic zero,
taking the quotient defines a map from $K_0^G(\Var_S)$ to $K_0(\Var_{S/G})$.
We will now show that such a map exists for more general $S$ and $G$.
To make sure that the quotient $X/G$ of a separated $S$-scheme of finite type $X$
with good $G$-action is a separated $S/G$-scheme of finite type,
we assume, as before, that $S/G$ is locally Noetherian and separated,
and that the quotient map $\pi_S:S \to S/G$ is finite.
For a point $s\in S/G$, we denote again by $F_s$ the residue field of $s$ and by $ G_s\subset G$
the stabilizer of a point $s'\in \pi ^{-1}(s)\subset S$.

Proving that the quotient map is well defined,
the main problem is to show that the second relation in Definition \ref{equivariant Gring} does not cause troubles,
hence we have to control quotients of $G$-equivariant affine bundles by affine $G$-actions.
In the proof of \cite[Lemma 3.2]{MR2106970} it is shown that, if the action of $G$ on $S$ is free,
the quotient of such a bundle is again an affine bundle.
Without the freeness assumption the quotient can be even singular,
hence in particular it is no an affine bundle in general.
But we have seen in Lemma \ref{lemma}  that in the modified Grothendieck ring,
or in the usual Grothendieck ring in the case of tame actions,
the class of the quotient of a
$G$-equivariant affine bundle only depends on its rank and its base.
This enables us to show the following theorem.
 
 \begin{thm}\label{quotientmap}
 Let $G$ be a finite abelian group.
 Assume that the residue field $F_s$ of any point $s\in S/G$ contains all $\lvert G_s\rvert $-th roots of unity.
 Then there is a well defined group homomorphism
 \[
 K^G_0(\Var_S)\to K_0^{\text{mod}}(\Var_{S/G})
 \]
 sending $[X]\in K_0^G(\Var_S)$ to $[X/G]\in K_0^{\text{mod}}(\Var_{S/G})$ for every $X\in (\Sch_{S,G})$.
 If the $G$-action on $S$ is tame,
 it factors through a group homomorphism
  \[
 K^G_0(\Var_S)\to K_0(\Var_{S/G}).
 \]
 \end{thm}

 \begin{proof} 
 For all $X\in (\Sch_{S,G})$,
 the quotient $X/G$ exists by \cite[Expos\'e V, Corollaire 1.4]{MR0238860},
 because the $G$-action on $X$ is assumed to be good.
 As the structure map $X\to S$ is $G$-equivariant,
 $X/G$ is an $S/G$-scheme.
 As $S/G$ is locally Noetherian and $S\to S/G$ is finite and hence of finite type,
 by \cite[Corollaire 1.5]{MR0238860},
 $X/G$ is of finite type over $S/G$.
 Moreover, the fact that $X$ is separated over $S/G$ implies
 the same for $X/G$.
 Hence we get a well defined map from $(\Sch_{S,G})$ to $K_0^{\text{mod}}(\Var_{S/G})$ and $K_0(\Var_{S/G})$, respectively.
So in order to show the proposition, we need to show that this map factors through $K^ G_0(\Var_S)$.
Hence we need to show that Relation (1) and Relation (2) from Definition \ref{equivariant Gring} still hold after taking the quotient.

Take any $X\in (\Sch_{S,G})$ and a closed subscheme $Y$ of $X$,
closed with respect to the action of $G$.
Set $U:=X\setminus Y$.
Denote by $\pi:X\to X/G$ the quotient map.
As $U \subset X$ is open and $G$-invariant,
by \cite[Expos\'e V, Corollaire 1.4]{MR0238860}, $\pi(U)\cong U/G$ and open in $X/G$.
By Lemma \ref{closed mod}, $[\pi(Y)]=[Y/G]\in K_0^{\text{mod}}(\Var_{S/G})$.
Hence we get, using the scissors relation in $K_0^{\text{mod}}(\Var_{S/G})$, that
\begin{equation}\label{eqrel1}
  [X/G]  
=[\pi(Y)]+[\pi(U)]
 =[Y/G]+[U/G] \in K_0^{\text{mod}}(\Var_{S/G}).
\end{equation}
If the action of $G$ on $S$, and hence also the action of $G$ on $X$ is tame,
then by Lemma \ref{rem tame decomposition} we get Equation~(\ref{eqrel1}) also in $K_0(\Var_{S/G})$.

It remains to show that Relation (2) still holds after taking the quotient.
Take any $B\in (\Sch_{S,G})$, and
let $V\rightarrow B$ and $W\to B$
be $G$-equivariant affine bundle of rank $d$ with affine $G$-action.
By Lemma \ref{lemma} we have that
\begin{equation}\label{eqrel2}
 [V/G]=\mathbb{L}^d_{S/G}[{B/G}]=[W/G]\in K_0^{\text{mod}}(\Var_{S/G}),
\end{equation}
 and if the action of $G$ on $S$ is tame,
 Equation (\ref{eqrel2}) holds in $K_0(\Var_S)$.
Altogether, taking the quotients gives us well defined maps as in the claim.
\end{proof}
 
 \begin{rem}
 By Remark \ref{rem decomposing}
 we only get a well defined quotient map with values in the modified Grothendieck
 ring of varieties in the case of a wild action on $S$,
even if we can show Lemma \ref{lemma}
 in the usual Grothendieck ring of varieties.
\end{rem}

\begin{rem} \label{necessary}
Let $S=\Spec(k)$ with trivial $G$-action for some group $G$.
Assume that there exists a finite Galois extension $K/k$ and a $d$-dimensional $K$-vector space
$V$ with a $k$-linear $G$-action lifting the Galois action,
such that $[V/G]\neq \mathbb{L}_k^{d}$ in $K_0^{\text{mod}}(\Var_k)$.
Note that $V\to \Spec(K)$ is an affine bundle of rank $d$ with affine $G$-action in the category $(\Sch_{S,G})$,
and the same holds for $\mathbb{A}^ {d}_K\to \Spec(K)$
with a $G$-action induced by the Galois action on $K$.
We have $\mathbb{A}^d_K/G\cong\mathbb{A}^d_k$, see Example \ref{extrivial}.
By definition $[V]=[\mathbb{A}^{d}_K]\in K_0^ G(\Var_k)$,
but $[V/G]\neq[\mathbb{A}^{d}_K/G]\in K_0^ G(\Var_k)$,
so there cannot be a well defined map from $K_0^ G(\Var_k)$ to $K_0^{\text{mod}}(\Var_k)$ as in
Theorem~\ref{quotientmap}.
Hence Remark \ref{exab} and Remark \ref{exHE} show that Theorem \ref{quotientmap} in general
does not hold without the assumption on the residue field $F_s$ for all $s\in S/G$ or for a non-abelian group $G$.
\end{rem}

  \begin{cor} \label{corM}
Notation and assumptions as in Theorem \ref{quotientmap}.
 Then there is a well defined group homomorphism
 \[
  \mathcal{M}_S^G\to \mathcal{M}^{\text{mod}}_{S/G}
 \]
sending $\mathbb{L}_S^{-s}[X]$ to  $\mathbb{L}_{S/G}^{-s}[X/{G}]$
for all $X\in (\Sch_{S,G})$.
If the action of $G$ on $S$ is tame,
we also get a well defined group homomorphism
$ \mathcal{M}_S^G\to \mathcal{M}_{S/G}$
 with this property.
 \end{cor}
 
 \begin{proof}
Using Theorem \ref{quotientmap} it suffices to show that for all $j\in \mathbb{Z}$ 
\begin{equation}\label{xmall}
 \mathbb{L}^{-j}_{S/G}[X/G]=\mathbb{L}_{S/G}^{-(j+1)}[(\mathbb{A}^1_{S}\times_S X)/G] \in \mathcal{M}_{S/G}
\end{equation}
and hence in $\mathcal{M}^{\text{mod}}_{S/G}$
for all $X\in (\Sch_{S,G})$.
Note that the fiber product of $X$ and $\mathbb{A}^1_S$ is taken in the category $(\Sch_{S,G})$.
Note that   
  $\mathbb{A}^1_{S}\times_SX =\mathbb{A}^1_{S/G}\times_{S/G}X $,
  and the action on $\mathbb{A}^ 1_{S/G}$ is trivial.
   As $\mathbb{A}^1_{S/G}$ is flat over $S/G$,
   \cite[Proposition 1.9]{MR0238860} implies that 
   $(\mathbb{A}^1_{S/G}\times_{S/G}X)/G=\mathbb{A}^1_{S/G}\times_{S/G}X/G$.
   Hence Equation (\ref{xmall}) holds.
\end{proof}

\noindent
  Consider now a profinite group
  \[
  \hat{G}=\lim_{\stackrel{\longleftarrow}{i\in I}} G_i.
  \]
 Let $S$ be a separated scheme with a $\hat{G}$-action factorizing through a good action of a finite quotient $G_j$ of $\hat{G}$.
 Assume that the quotient $S/\hat{G}=S/G_j$ is a separated, locally Noetherian scheme,
 and that the quotient map $\pi_S: S\to S/\hat{G}$ is finite.
 For all $s\in S$ and $i\geq j$, set $q_{si}:=\lvert G_{si}\rvert $,
 with $G_{si}\subset G_i$ the stabilizer of a point $s'\in \pi_s^{-1}(s)\subset S$
 under the action of the $G_i$.

 \begin{cor}\label{corprofinite}
 Let $\hat{G}$ be a profinite abelian group.
 Assume that the residue field $F_S$ of any point $s\in S/\hat{G}$ contains all $q_{si}$-th roots of unity for all $i\geq j$.
 Then there are well defined group homomorphisms
 \[
  K_0^{\hat{G}}(\Var_S)\to K_0^{\text{mod}}(\Var_{S/\hat{G}})\text{ and } \mathcal{M}_S^{\hat{G}} \to \mathcal{M}^{\text{mod}}_{S/\hat{G}}
 \]
sending the class of a separated scheme $X$
to the class of its quotient $X/\hat{G}$.
If the action of $\hat{G}$ on $S$ is tame, we get well defined group homomorphisms
 \[
  K_0^{\hat{G}}(\Var_S)\to K_0(\Var_{S/\hat{G}})\text{ and } \mathcal{M}_S^{\hat{G}} \to \mathcal{M}_{S/\hat{G}}
 \]
 with this property.
 \end{cor}
 
 \begin{proof}
 By assumption $S/\hat{G}=S/G_j=S/G_i$ for all $i\geq j$.
 Hence
by Theorem~\ref{quotientmap}
there is a well defined map $K_0^ {G_i}(\Var_S)\to K_0^{\text{mod}}(\Var_{S/G_i})=K_0^{\text{mod}}(\Var_{S/\hat{G}})$ sending $X\in (\Sch_{S,G_i})$ to its quotient
$X/G_i$ for all $i\geq j$.
Using this map we get the following commutative diagram:
\[
 \xymatrix{
\dots \ar[r] & K_0^ {G_i}(\Var_S) \ar[r] \ar[rd] & K_0^ {G_{i+1}}(\Var_S)\ar[r] \ar[d] & \dots  \\
&  & K_0^{\text{mod}}(\Var_{S/\hat{G}})  & 
 }
\]
Here the maps in the first line are given as in Remark \ref{remark}.
Hence we get an induced map
\[ \lim_{\stackrel{\longrightarrow}{i\in I}}K_0^{G_i}(\Var_S) \to
   K_0^{\text{mod}}(\Var_{S/\hat{G}})
\]
with the required property.
The tame case works analog.
The statement for $\mathcal{M}_S^{\hat{G}}$
follow with a similar argument from Corollary \ref{corM}.
 \end{proof}

\begin{rem}\label{K0-linear}
 Let $k$ be a field with trivial tame $G$-action.
 Then we can view $K_0^G(\Var_k)$ as a module over $K_0(\Var_k)$
 by mapping the class of a $k$-scheme of finite type $X$ in $K_0(\Var_k)$ to the class of $X$ with trivial $G$-action.
 It is clear that the quotient map is trivial on the image of $K_0(\Var_k)$ in $K_0^G(\Var_k)$.
 As every $k$-scheme of finite type is flat over the field $k$,
 it follows from \cite[Proposition 1.9]{MR0238860}
 that the quotient map maps the class of $X\times_k V$ in $K_0^G(\Var_k)$ to $[X]\cdot[V/G]\in K_0(\Var_k)$.
 Hence it follows that the quotient map
 \[
  K_0^G(\Var_k)\to K_0(\Var_k)
 \]
is $K_0(\Var_k)$-linear.
Similarly we get that the quotient map $\mathcal{M}_k^G\to \mathcal{M}_k$ is $\mathcal{M}_k$-linear.
Moreover we get the analog statements for a profinite group $\hat{G}$.
\end{rem}

 \section{The quotient of the motivic nearby fiber}
 \label{an application}

 \noindent
Throughout this section, 
if not mentioned otherwise,
let $k$ be a field of characteristic zero
containing all roots of unity,
let $X$ be an irreducible algebraic variety over $k$, and
let $f:X\to \mathbb{A}_k^1$ be a non-constant morphism.
We denote by $X_0\subset X$ the zero locus of $f$ in $X$,
and assume that $X_\eta:=X \times_{\mathbb{A}_k^1}\mathbb{A}_k^1\setminus \{0\}=X\setminus X_0$ is smooth.

\medskip

We are now going to use the results from the previous section
to show that the quotient of the motivic nearby fiber is a well defined invariant with values in $\mathcal{M}_{X_0}$.
The motivic nearby fiber can be attached to a map $f: X\to \A^1_k$ as above.
It was constructed in \cite{DL3}
as a limit of the motivic Zeta function,
and was investigated in more details in \cite{MR2106970}.
Moreover we show that modulo $\mathbb{L}$, this quotient is equal to motivic reduction $R(f)$ in the image of $K_0(\Var_{X_0})$ in $\mathcal{M}_k$,
see Proposition~\ref{application}.
Here $R(f)$ is the class of the special fiber of a smooth modification of $f$
in  $K_{0}(\Var_{X_0})/\mathbb{L}$.
A \emph{smooth modification} of $f: X\to \mathbb{A}^1_k$ is 
a smooth irreducible algebraic variety $Y$ over $k$,
together with a proper morphism 
$h:Y\to X$ such that the restriction 
$h: Y\setminus h^ {-1}(X_0)\to X\setminus X_0$ is an isomorphism.
Due to weak factorization, the definition of the motivic reduction
does not depend on the choice of such a modification,
see Lemma \ref{welldef R(f)}.

\medskip

This result implies the following: if $f: X\to \mathbb{A}_k^1$ is proper
and $X$ is smooth, and the generic fiber $X_\eta$ of $f$ is equal to $1$ modulo $\mathbb{L}$ in $K_0(\Var_{\A_k^1\setminus \{0\}})$,
then the same holds for the special fiber of $f$ in the image of $K_0(\Var_k)$ in $\mathcal{M}_k$,
see Corollary \ref{1modL}.
This can be seen as a motivic analog of the main theorem in \cite[Theorem 1.1]{MR2247971},
which says that
if $V$ is an absolutely irreducible smooth projective variety
over a local field $K$ with finite residue field $F$ such that the $m$-th \'etale cohomology
of $V\times_K \bar{K}$, $\bar{K}$ the algebraic closure of $K$, has
coniveau $1$
for all $m\geq1$,
then the number of rational points of the special fiber of any
regular projective model of $V$ is congruent $1$ modulo $\lvert F \rvert$.

How does this analogy work?
In both cases one deduces a property of the special fiber from a property of the generic fiber of some proper and smooth scheme.
We are now going to outline the connection between the properties on the generic fibers and the properties on the special fibers, respectively.

As remarked in \cite{MR2247971}, if the characteristic of $K$
is equal to zero,
then the fact that the \'etale cohomology of $V$ has coniveau $1$
implies that the Hodge type of the de Rham cohomology
is $\geq 1$,
or equivalently that $H^q(V,\mathcal{O}_V)=0$ for all $q\geq 1$.
Now
consider the Hodge-Deligne polynomial $HD: K_0(\Var_K)\to \mathbb{Z}[u,v]$, see \cite[Example~4.1.6]{MR2885336}.
This is a ring morphism sending
the class of a projective and smooth $K$-variety $X$ to $\sum_{p,q} (-1)^{p+q} \Dim_K(H^q(X,\Omega_X^p))u^pv^q$.
We have that 
\[
HD(\mathbb{L})=HD([\mathbb{P}_K^1])-HD(1)=1+uv-1=uv.
\]
Hence, if the class of $V$ is equal to $1$ modulo $\mathbb{L}$ in $K_0(\Var_K)$,
 $HD(V)=1$ modulo $uv$, hence in particular $H^q(V,\mathcal{O}_V)=0$ for $q\geq 1$, so
the de Rham cohomology of $V$ has Hodge type $\geq 1$.

Note furthermore that if we have a finite field $F$,
and the class of a variety $V_F$ over $F$ in $K_0(\Var_F)$
is equal to $1$ modulo $\mathbb{L}$,
then $\#(V_F(F))=1$ modulo $\lvert F \rvert$.

\subsubsection*{The Motivic nearby fiber}
Recall the following notation from \cite{DL3}.
Let ${h: Y \to X}$ be an \emph{embedded resolution} of $f$, 
i.e., $h$ is a proper morphism inducing an isomorphism $Y\setminus h^{-1}(X_0)\to X\setminus X_0$, $Y$ is smooth, and $h^ {-1}(X_0)$ is a simple normal crossing divisor.
Such an embedded resolution always exists due to resolution of singularities in characteristic zero.
Denote by $E_i$, $i\in J$, the irreducible components of $h^{-1}(X_0)$.
For each $i\in J$, denote by $N_i$ the multiplicity of $E_i$ in the divisor $f\circ h$
on $Y$.
For $I\subset J$,
we consider non-singular varieties 
\[
E_I=\bigcap _{i\in I}E_i, \text{ and } E_I^o:=E_I\setminus \bigcup_{j\in J\setminus I} E_j.
\]
Note that $\bigcup_{\emptyset \neq I\subset J}E_I^o=h^{-1}(X_0)$.

Let $\mu_n\subset \mathbb{C}$ be the group of $n$-th roots of unity for all $n\in \mathbb{N}.$ 
Set ${m_I:=\text{gcd}(N_i)_{i\in I}}$.
We introduce an unramified Galois cover $\tilde{E}_I^o$ of $E_i^o$ with Galois group 
$\mu_{m_I}$ as follows.
Let $U$ be an affine Zariski open subset of $Y$,
such that, on $U$, $f\circ h=uv^{m_i}$,
with $u$ a unit on $U$ and $v$ a morphism from $U$ to $\mathbb{A}_k^1$.
Then the restriction of $\tilde{E}_I^o$ above $E_I^o\cap U$,
denoted by $\tilde{E}_I^o\cap U$, is defined as 
\[
 \{(z,y)\in \mathbb{A}_k^1\times(E_I^o\cap U)\mid z^{m_I}=u^{-1}\}.
\]
Note that $E_I^o$ can be covered by such affine open subset $U$ of $Y$.
Gluing together the $\tilde{E}_I^o\cap U$, in the obvious way,
we obtain the cover $\tilde{E}_I^o$ of $E_I^o$ which has a natural $\mu_{m_I}$-action
(obtained by multiplying the $z$-coordinate with the elements of $\mu_{m_I}$).
This $\mu_{m_I}$-action on $\tilde{E}_I^o$ induces a $\hat{\mu}=\lim \mu_i$-action on 
$\tilde{E}_I^o$ in the obvious way.
Note that by construction $\tilde{E}_I^o/\hat{\mu}\cong E_I^o$.

\begin{defn}\label{dfn mnf}
\cite[Definition 3.5.3]{DL3}
 With this notation the \emph{motivic nearby fiber} is given by
\[
 \mathcal{S}_f:=\sum_{\emptyset \neq I\subset J} (1- \mathbb{L})^{\lvert I \rvert-1}[\tilde{E}^o_{I}]\in \mathcal{M}_{X_0}^{\hat{\mu}}.
\]
This definition does not depend on the choice of $h:Y\to X$, see \cite[3.3.1]{DLLefschetz}.
\end{defn}
\noindent
As by Corollary~\ref{corprofinite},
there is a well defined map $\mathcal{M}_{k}^ {\hat{\mu}}\to \mathcal{M}_k$
sending the class of a variety with $\hat{\mu}$-action to its quotient,
the quotient of the motivic nearby fiber
 \begin{equation}\label{Sf/G}
   \mathcal{S}_f/{\hat{\mu}}:=\sum_{\emptyset \neq I\subset J} (1- \mathbb{L})^{\lvert I \rvert-1}[\tilde{E}^o_{I}/\hat{\mu}]=\sum_{\emptyset \neq I\subset J} (1- \mathbb{L})^{\lvert I \rvert-1}[{E}^o_{I}]
 \in \mathcal{M}_{X_0} 
 \end{equation} 
 is a well defined invariant of $f:X\to \mathbb{A}_k^1$.
 In particular it does not depend on the choice of $h: Y\to X$.

By \cite[Definition 8.1]{MR2106970},
there is a well defined \emph{nearby cycle morphism}
\[
 \psi : \mathcal{M}_{\A_k^1}\to \mathcal{M}_k^{\hat{\mu}}
\]
sending the class of $X\in \mathcal{M}_{\mathbb{A}_k^1}$ to $\mathcal{S}_f$ for every smooth $k$-variety $X$ with a proper map $f: X\to \A^1_k$.
Here $\mathcal{S}_f$ is the image of $\mathcal{S}_f$
under the map $\mathcal{M}_{X_0}^{\hat{\mu}}\to \mathcal{M}_k^{\hat{\mu}}$
induced  by the structure map $X_0\to \Spec(k)$.
This morphism is $\mathcal{M}_k$-linear,
and $1=[\A^1_k]\in \mathcal{M}_{\A_k^1}$ is mapped to $1=[\Spec(k)]\in \mathcal{M}_k^{\hat{\mu}}$.
Using Corollary \ref{corprofinite} we get a well defined group morphism
\[
 \bar{\psi}: \mathcal{M}_{\A^1_k}\to \mathcal{M}_k
\]
sending the class of $X$ to the class of $\mathcal{S}_f/\bar{\mu}$ for every smooth $k$-variety $X$ with a proper map $f: X\to \A^1_k$.
Again $1\in \mathcal{M}_{\A_k^1}$ is mapped to $1\in \mathcal{M}_k$, and using Remark~\ref{K0-linear}
we see that $\bar{\psi}$ is $\mathcal{M}_k$-linear, too.

By \cite[List of properties 8.4]{MR2106970} the image of every $\mathbb{A}_k^1$-scheme
supported in the point $0$ is trivial.
So we get in fact maps
\begin{equation}\label{PSI}
  \psi: \mathcal{M}_{\A_k^1\setminus \{0\}}\to \mathcal{M}_k^{\hat{\mu}}
 \text{ and }
  \bar{\psi}: \mathcal{M}_{\A^1_k\setminus \{0\}}\to \mathcal{M}_k
\end{equation}
sending a smooth variety $X_\eta$ over $\mathbb{A}_k^1\setminus \{0\}$ to $\mathcal{S}_f$ and $\mathcal{S}_f/\hat{\mu}$, respectively,
for some proper map $f:X\to \mathbb{A}_k^1$
with $X$ smooth and irreducible, and $X_\eta \cong X\times _{\mathbb{A}_k^1}\mathbb{A}_k^1\setminus \{0\}$.
This maps are also $\mathcal{M}_k$-linear and map $1$ to $1$.

\begin{rem}
Let $X$ be a smooth and proper variety over $K=k(\!(t)\!)$.
Let $\mathcal{X}$ be proper sncd-model of $X$,
i.e., $\mathcal{X}$ is a proper and flat $k[\![t]\!]$-scheme with generic fiber isomorphic to $X$
such that its special fiber $X_0$
is a simple normal crossing divisor.
By \cite[Definition~8.3]{NiSe}
the motivic volume $S(\hat{X},\hat{K})$
with values in $\mathcal{M}_k$,
which we can associate with the $t$-adic completion $\hat{X}$ of $\mathcal{X}$,
does only depend on the generic fiber of $\hat{X}$,
and thus only on $X$ and not on the model $\mathcal{X}$.
By \cite[Proposition~8.2]{NiSe}
we have that
\[
 S(\hat{X},\hat{K})=\mathbb{L}^{-m}\sum_{\emptyset \neq I\subset J} (1- \mathbb{L})^{\lvert I \rvert-1}[\tilde{E}^o_{I}]\in \mathcal{M}_k,
\]
where the $\tilde{E}_I^o$ are constructed from $X_0$
as done above.
If $\mathcal{X}$ is actually coming from some $f: X\to \mathbb{A}^1_k$
as studied above, by \cite[Theorem 9.13]{NiSe},
we have that
\[
 S(\hat{X},\hat{K})=\mathbb{L}^{-(m-1)}\mathcal{S}_f\in \mathcal{M}_k.
\]
If we could show that
actually $S(\hat{X},\hat{K})$ was well defined in $\mathcal{M}_k^{\hat{\mu}}$ (this is work in progress),
where we consider the $\tilde{E}^o_I$ with $\hat{\mu}$-action as done above,
then also
\[
 S(\hat{X},\hat{K})/\hat{\mu}=\mathbb{L}^{-m}\sum_{\emptyset \neq I\subset J} (1- \mathbb{L})^{\lvert I \rvert-1}[{E}^o_{I}]\in\mathcal{M}_k
\]
would hold by Corollary \ref{corprofinite},
and we could proof results analogous to Proposition~\ref{application} and Corollary \ref{1modL} also in this context.
\end{rem}

\subsubsection*{Connection with the motivic reduction}
To be able to define the motivic reduction of $f$, we first need to proof the following lemma.

\begin{lem}\label{welldef R(f)}
Let $h: Y \to X$ be any smooth modification of $f:X\to \mathbb{A}_k^1$.
Then  the class of $h^ {-1}(X_0)$ in $K_0(\Var_{X_0})/\mathbb{L}$ does not depend on the choice of $h$.
\end{lem}

\begin{proof}
Let $h_1: Y_1\to X$, $h_2: Y_2\to X$ be two smooth modifications of $f$.
We have shown that 
\begin{equation}\label{hi}
 [h_1^{-1}(X_0)]=[h_2^{-1}(X_0)]\in K_0(\Var_{X_0})/\mathbb{L}.
\end{equation}
Consider the fiber product $Y_{12}:=Y_1\times_XY_2$.
Note that the projection maps ${p_i: Y_{12}\to Y_i}$ are proper,
and induce isomorphisms between $ Y_{12}\setminus p_i^{-1}(h_i^{-1}(X_0))$ and $Y_i\setminus h_i^{-1}(X_0)$.
As the $h_i$ are smooth modifications of $f$,
the $Y_i\setminus h_i^{-1}(X_0)$ are isomorphic to $X_\eta =X\setminus X_0$, which is smooth by assumption.

Let $b: \tilde{Y} \to Y_{12}$ be an embedded resolution of 
$f\circ h_1 \circ p_1=f\circ h_2\circ p_2: Y_{12}\to \mathbb{A}_k^1$.
Set $g_i:=p_i\circ b$.
By construction $g_1^{-1}(h_1^{-1}(X_0))=g_2^{-1}(h_2^{-1}(X_0))$.
Hence in order to show Equation (\ref{hi}),
it suffices to show that $[h_i^{-1}(X_0)]=[g_i^{-1}(h_i^{-1}(X_0))]$ in $ K_0(\Var_k)/\mathbb{L}$.
By the symmetry of the construction it suffices to show it for $i=1$.

Note that $g_1: \tilde{Y}\to Y_1$ is a proper birational morphism over $X$ between two smooth varieties.
By \cite[Remark 2.3]{MR2059227}, the Weak Factorization 
Theorem, see \cite[Theorem 0.1.1]{MR1896232},
holds for $g_1$, i.e., there exist a sequence of birational maps as follows
\[
  \xymatrix{
X_1=V_0\ar@{-->}[r]^{\phi_1}\ar@/_5mm/[rrrrr]_{g_1}& V_1\ar@{-->}[r]^{\phi_2}& V_2 \ar@{-->}[r]^{\phi_2}&\dots \ar@{-->}[r]^{\phi_{l-1}} &  V_{l-1}\ar@{-->}[r]^{\phi_{l}}& V_l=Y_i,
 }
\]
and for all $i$, either $\phi_i:V_{i-1}\dashrightarrow V_{i}$ or
$\phi_i^{-1}: V_i\dashrightarrow V_{i-1}$
is a morphism obtained by blowing up a smooth center.
By \cite[Remark 2.4]{MR2059227}, the factorization is a factorization over $X$,
i.e., there are structure maps $\varphi_i: V_i\to X$, with $\varphi_0=h_1\circ g_1$ and $\varphi_l=h_1$,
and the $\phi_i$
are maps over $X$.

Take any $i\in \{1,\dots,l\}$.
If $\phi_i: V_{i-1}\to V_i$ is a blowup in the smooth center $C_i\subset V_i$,
then we get using the scissors relation that
\begin{align*}
 [\varphi_{i-1}^ {-1}(X_0)]&=[\phi_i^{-1}(\varphi_i^{-1}(X_0))]=[\phi_i^{-1}(\varphi_i^ {-1}(X_0)\setminus C_i)]+[\phi_i^{-1}(\varphi_i^{-1}(X_0)\cap C_i)]\\
 &=[\varphi_i^ {-1}(X_0)\setminus C_i]+[\mathbb{P}_{X_0}^{\Codim_{V_i}(C_i)}][C_i\cap \varphi_i^{-1}(X_0)]\\
 &=[\varphi_i^ {-1}(X_0)\setminus C_i]+[C_i\cap \varphi_i^{-1}(X_0)]=[\varphi_i^ {-1}(X_0)] \in K_0(\Var_{X_0})/(\mathbb{L}).
\end{align*}
Analogously, we get the same statement if $\phi_i^{-1}$ is a blowup. 
Hence it follows that the classes of $g_1^ {-1}(h_1^{-1}(X_0))$ and $h_1^ {-1}(X_0)$ coincide in $K_0(\Var_k)/(\mathbb{L})$,
and hence the claim follows as observed above.
\end{proof}

\begin{defn}\label{dfn mr}
 The \emph{motivic reduction} $R(f)$ of $f: X\to \mathbb{A}_k^1$ is defined as the class of $h^{-1}(X_0)$
  in $K_0(\Var_{X_0})/\mathbb{L}$ of any smooth modification $h:Y\to X$ of $f$.
\end{defn}

\begin{nota}
Let $R\in K_0(\Var_{X_0})$ any  element in the inverse image of $R(f)$ under the quotient map.
We denote with $R(f)$ also the class of the image of $R$ in $\mathcal{M}_{X_0}$ modulo $\mathbb{L}$.

Using the forgetful map $K_0(\Var_{X_0})\to K_0(\Var_k)$
induced by the structure morphism $X_0\to \Spec(k)$,
we can view $R(f)$ also as an element in $K_0(\Var_k)/\mathbb{L}$.
We denote with $R(f)$ also its image in $K_0(\Var_k)/\mathbb{L}$.
\end{nota}

\noindent
Now we can combine the definition of the motivic reduction with the motivic nearby fiber,
and get the following proposition.

\begin{prop}\label{application} 
The class of $R(f)$ and $S_f/\hat{\mu}$ in the image of $K_0(\Var_{X_0})$ in $\mathcal{M}_{X_0}$ modulo $\mathbb{L}$ coincide.
 \end{prop}

\begin{proof}
Let $h: Y\to X$ be a embedded resolution of $f$, and
let $E_I^o$ be constructed from $h: Y\to X$ as done above.
Then by Equation (\ref{Sf/G}) we have
\begin{align*}
 \mathcal{S}_f/{\hat{\mu}}
& =\sum_{\emptyset \neq I\subset J}[{E}^o_{I}] + \mathbb{L} \Big( \sum_{\emptyset \neq I\subset J} \sum_{k=1}^{\lvert I \rvert-1} {\lvert I \rvert-1 \choose k}\mathbb{L}^{k-1}[E_I^o]\Big)\\
& =[h^{-1}({X}_0)]+\mathbb{L} \Big(\sum_{\emptyset \neq I\subset J} \sum_{k=1}^{\lvert I \rvert-1} {\lvert I \rvert-1 \choose k}\mathbb{L}^{k-1}[E_I^o]\Big) \in \mathcal{M}_{X_0}.
\end{align*}
One observes that $\mathcal{S}_f/{\hat{\mu}}$ lies in the image of $K_0(\Var_{X_0})$.
Hence 
 $\mathcal{S}_f/{\hat{\mu}}$ is equal to $[h^{-1}({X}_0)]$ modulo $\mathbb{L}$
 in the image of $K_0(\Var_{X_0})$ in $\mathcal{M}_{X_0}$.
This is equal to $R(f)$, because $h:Y\to X$ is an embedded resolution and hence a smooth modification 
of $f: X\to \mathbb{A}_k^1$.
\end{proof}

 \begin{cor}\label{1modL}
 Let $X$ be a smooth variety over $k$, and let ${f: X\to \mathbb{A}_k^1}$ be a proper morphism.
 If the class of $X_\eta$ is equal to $1$ modulo $\mathbb{L}$
 in $K_0(\Var_{\A^1_k\setminus \{0\}})$,
 then the class of $f^{-1}(X_0)$ is equal to $1$ modulo $\mathbb{L}$
 in the image of $K_0(\Var_k)$ in $\mathcal{M}_k$.
 \end{cor}
 
\begin{proof}
 If $[X_\eta]=1 \mod \mathbb{L}\in K_0(\Var_{\A^1_k\setminus\{0\}})$,
 we can write $[X_\eta]=1+\mathbb{L}[V]$ with $[V]\in  K_0(\Var_{\A^1_k\setminus\{0\}})$,
 also for the class of $X_\eta$ in $\mathcal{M}_{\A^1_k\setminus \{0\}}$.
 Consider the map $\bar{\psi}: \mathcal{M}_{\A^1_k\setminus \{0\}}\to \mathcal{M}_k$, see Equation (\ref{PSI}).
 
 On the one hand side,
  $\bar{\psi}([X_\eta])=1+\mathbb{L}\bar{\psi}(V)\in \mathcal{M}_k$, because
 $\bar{\psi}$ is $\mathcal{M}_k$-linear and maps $1$ to $1$.
 On the other hand,
$\bar{\psi}([X_\eta])=\mathcal{S}_f/\hat{\mu}$, and we have already seen that this is equal to $R(f)$
 in the image of $K_0(\Var_k)$ in $\mathcal{M}_k$.
Here $\mathcal{S}_f/\hat{\mu}$ and $R(f)$ 
 are elements in $K_0(\Var_k)$ via the map $K_0(\Var_{X_0})\to K_0(\Var_k)$ induced by the structure map $X_0\to \Spec(k)$.
 Moreover $R(f)=[f^{-1}(X_0)]$, because $X$ is smooth, and hence $\Id$ is a smooth modification of $f$.

All together $[f^{-1}(X_0)]$ is equal to $1$ modulo $\mathbb{L}$ in the image of $K_0(\Var_k)$ in $\mathcal{M}_k$.
\end{proof}

\begin{rem}
 It should be possible to proof Corollary \ref{1modL} without using $S_f$,
 by showing that there exists a
 well defined map
 $K_0(\Var_{\mathbb{A}_k^1})\to K_0(\Var_{k})/\mathbb{L}$, which is $K_0(\Var_k)$-linear, 
and sends $[X]$ to $R(f)$ if the structure map $f:X\to \mathbb{A}_k^1$ is proper and non-constant,
 and to $0$ if $f$ is proper and constant.
 To do this, one would need to consider the blowup relations
 in the Grothendieck ring of varieties,
 see \cite[Theorem 5.1]{MR2059227}.
\end{rem}


 \begin{rem}
 By Remark \ref{rem 0div neg},
 $K_0(\Var_{X_0})$ is not a subgroup of $\mathcal{M}_{X_0}$, hence
we cannot show Proposition \ref{application} in $K_0(\Var_k)/\mathbb{L}$.
 If we could show that the motivic nearby fiber was well defined
 in $K_0^{\hat{\mu}}(\Var_{X_0})/\mathbb{L}$, it would follow from
the fact that taking the quotient also
gives a well defined map from $K_0^{\hat{\mu}}(\Var_{X_0})/(\mathbb{L})$ to $K_0(\Var_{X_0})/(\mathbb{L})$,
that $\mathcal{S}_f/\hat{\mu}$ would be well defined in $K_0(\Var_{X_0})/(\mathbb{L})$.
Like this we could also show that $R(f)$ is well define without using the Weak Factorization Theorem.
\end{rem}

\noindent
In order to avoid the problem that $K_0(\Var_{X_0})$ is maybe not a subgroup of $\mathcal{M}_{X_0}$,
we can work in the Grothendieck ring of effective motives.
To make things easier we work over $k$ right away.

Let $\text{Mot}^{\text{eff}}_{k}$ be the additive category of effective motives with rational coefficients,
let $K_0(\text{Mot}^{\text{eff}}_{k})$ its Grothendieck ring, and let
$\mathbb{L}_{\text{mot}}$ be the class of the Lefschetz motive in this ring.
Let $\text{Mot}_{k}$ be the category of motives with rational coefficients,
let $K_0(\text{Mot}_{k})$ its Grothendieck ring, and let
$\mathbb{L}_{\text{mot}}$ be the image of the Lefschetz motive.
For the precise definitions of these objects we refer to \cite[Chapter~4]{MR2115000}.
The notation used here can be found for example in \cite[Example 2.3]{MR2885336}.
By \cite[Theorem 4.11]{MR2885336} 
we get commuting maps as follows:
\[
 \xymatrix{
 K_0(\Var_k)\ar[rr]^{\chi^{\text{eff}}_{\text{mot}}} \ar[d] & & K_0(\text{Mot}^{\text{eff}}_{k})\ar[d]^{\rho}\\
\mathcal{M}_k \ar[rr]^{\!\!\!\!\!\!\!\!\!\!\!\!\!\!\!\!\!\!\!\!\!\!\!\!\!\!\!\!\!\!\!\!\!\!\!\!\chi_{\text{mot}}} & & K_0(\text{Mot}^{\text{eff}}_{k})[\mathbb{L}_{\text{mot}}^{-1}]\cong K_0(\text{Mot}_{k})
}
\]
Here $\chi^{\text{eff}}_{\text{mot}}$
maps the class of a projective $k$-variety $X$
to the class of the effective motive $(X,\Id)$.
$\mathbb{L}$ is mapped to $\mathbb{L}_{\text{mot}}$.
By \cite[Proposition 2.7]{MR2770561}
$\rho$ is injective if one assumes the following standard conjecture,
which can be found for example in \cite[Conjecture 2.5]{MR1879805}:

\begin{conj}\label{con}
If $M$ and $N$
are objects in $\mathrm{Mot}^ {\mathrm{eff}}_k$,
then $[M]=[N]\in K_0(\mathrm{Mot}^ {\mathrm{eff}}_k)$
if and only if $M$ and $N$
are isomorphic.
\end{conj}

\noindent
As shown in \cite[Proposition 4.4]{MR3058610}, Conjecture \ref{con} holds if M an N are supposed to be finite dimensional.
An important conjecture by Kimura and O'Sullivan predicts that all the motives
$M\in \text{Mot}^{\text{eff}}_k$ are finite dimensional. See \cite[Conjecture 2.7]{MR2167204}, or \cite[Conjecture KS, page 390]{MR3058610} for this very precise formulation.

Assume now that Conjecture \ref{con} is true.
Hence $\rho$ is injective,
and as $\mathcal{S}_f/\hat{\mu}$ lies in the image of $K_0(\Var_k)$ in $\mathcal{M}_k$,
the inverse image of $\chi_{\text{mot}}(\mathcal{S}_f/\hat{\mu})$ under $\rho$
has precisely one element,
which we denote by $\mathcal{S}_f/\hat{\mu}_{\text{mot}}$.
Set $R(f)_{\text{mot}}:=\chi_{\text{mot}}^{\text{eff}}(R(f))$.
Proposition \ref{application} then implies the following:

\begin{cor}\label{application motives}
$\mathcal{S}_f/\hat{\mu}_{\text{mot}}$ and $R(f)_{\text{mot}}$
coincide in $ K_0(\text{Mot}^{\text{eff}}_{k})$.
\end{cor}

\section*{Acknowledgments}
\noindent
During the research for this article, I was supported by a research fellowship
 of the \textbf{DFG} (Aktenzeichen HA 7122/1-1).
I thank H\'el\`ene Esnault for discussing her result with me.
Moreover, I am very thankful to Johannes Nicaise for the numerous discussions we had, for the ideas he gave me and for the suggestions he made. 

 \bibliographystyle{babalpha}

	\bibliography{refegr}

\end{document}